\documentclass[11pt,oneside]{amsart}

\usepackage{amsmath,ifthen, amsfonts, amssymb,
srcltx, amsopn, color}

\usepackage{graphicx}

\newcommand{\showcomments}{yes}

\newsavebox{\commentbox}
{\ifthenelse{\equal{\showcomments}{yes}}%
{\footnotemark
        \begin{lrbox}{\commentbox}
        \begin{minipage}[t]{1.25in}\raggedright\sffamily\tiny
        \footnotemark[\arabic{footnote}]}
{\begin{lrbox}{\commentbox}}}%
{\ifthenelse{\equal{\showcomments}{yes}}%
{\end{minipage}\end{lrbox}\marginpar{\usebox{\commentbox}}}
{\end{lrbox}}}

\newtheorem{thm}{Theorem}[section]
\newtheorem{lem}[thm]{Lemma}

\newtheorem{cor}[thm]{Corollary}

\newtheorem{prop}[thm]{Proposition}

\newtheorem*{thmA}{Theorem~\ref{thm:HGisQT}}

\newtheorem*{thmE}{Theorem~\ref{thm:flatplaneonlyif}}
\newtheorem*{corB}{Corollary~\ref{cor:weakrelredux}}
\newtheorem*{corC}{Corollary~\ref{cor:asdimvertstab}}

\theoremstyle{definition}
\newtheorem{defn}[thm]{Definition}
\newtheorem{rem}[thm]{Remark}
\newtheorem{exmp}[thm]{Example}

\newtheorem{claim*}{Claim}

\DeclareMathOperator{\dimension}{dim}

\DeclareMathOperator{\rank}{rk}

\DeclareMathOperator{\diam}{diam}
\DeclareMathOperator{\full}{Full}
\DeclareMathOperator{\asdim}{asdim}

\DeclareMathOperator{\comp}{c}

\newcommand{\scname}[1]{\text{\sf #1}}
\newcommand{\area}{\scname{Area}}
\newcommand{\coll}{\;\;\makebox[0pt]{$\bot$}\makebox[0pt]{$\smile$}\;\;}

\newcommand{\field}[1]{\mathbb{#1}}
\newcommand{\integers}{\ensuremath{\field{Z}}}

\newcommand{\naturals}{\ensuremath{\field{N}}}
\newcommand{\reals}{\ensuremath{\field{R}}}

\DeclareMathOperator{\ancestor}{Ancestor}

\let\oldmarginpar\marginpar
\renewcommand\marginpar[1]{\-\oldmarginpar[\raggedleft\footnotesize #1]%
{\raggedright\footnotesize #1}}

\begin{document}

\title{Weak hyperbolicity of cube complexes and quasi-arboreal groups}
\author[Mark~F.~Hagen]{Mark F. Hagen}
           \address{Dept. of Math.\\
                    University of Michigan \\
                    Ann Arbor, Michigan, USA}
           \email{markfhagen@gmail.com}


\date{\today}
\date{\today}
\maketitle

\begin{abstract}
We examine a graph $\Gamma$ encoding the intersection of hyperplane carriers in a CAT(0) cube complex $\widetilde X$.  The main result is that $\Gamma$ is quasi-isometric to a tree.  This implies that a group $G$ acting properly and cocompactly on $\widetilde X$ is weakly hyperbolic relative to the hyperplane stabilizers.  Using Wright's recent result on the aymptotic dimension of CAT(0) cube complexes, we give a generalization of a theorem of Bell and Dranishnikov on the finite asymptotic dimension of graphs of asymptotically finite-dimensional groups.  Finally, we apply contact graph techniques to prove a cubical version of the flat plane theorem stated in terms of complete bipartite subgraphs of $\Gamma$.
\end{abstract}

\section{Introduction}\label{sec:introduction}
The set $\mathcal W$ of hyperplanes in a CAT(0) cube complex $\widetilde X$ admits a \emph{crossing} relation and, more generally, a \emph{contact} relation: distinct hyperplanes $W_1,W_2\in\mathcal W$ \emph{contact} if they have dual 1-cubes $c_1,c_2$ that have a common 0-cube.  In particular, $W_1$ and $W_2$ contact if they \emph{cross}, which happens when $c_1$ and $c_2$ form the corner of a 2-cube.  The contact relation is encoded in a \emph{contact graph} $\Gamma$, whose vertex set is $\mathcal W$ and whose edges correspond to contacting pairs of hyperplanes.  The crossing relation gives a \emph{crossing graph} $\Delta\subseteq\Gamma$, with the same vertex set, whose edges correspond to crossing pairs of hyperplanes.

The goal of this paper is to describe some properties of the contact graph and illustrate some uses of the contact graph and disc diagrams in studying CAT(0) cube complexes and cubulated groups.  A geometric advantage of examining the contact graph is that, unlike the crossing graph, it is always connected.  Moreover, in Section~\ref{sec:quasitree}, we prove the following:

\begin{thmA}
The contact graph $\Gamma$ associated to a CAT(0) cube complex $\widetilde X$ is quasi-isometric to a tree.
\end{thmA}

Hence cubulating a group entails construction of an action on a quasi-tree.  Theorem~\ref{thm:HGisQT} can be deduced from Manning's ``bottleneck'' condition characterizing quasi-trees~\cite{ManningPseudocharacters}; we also give a more constructive proof using disc diagram techniques, by constructing a \emph{graded root tree} $\mathcal T$ and exhibiting a quasi-isometry $\Gamma\rightarrow\mathcal T$ which grades the hyperplanes by the distances of their images to a specified base vertex in $\mathcal T$.

Farb introduced the notion of \emph{weak hyperbolicity} of a group $G$ relative to a collection of subgroups $\{P\}$, to mean that the metric space obtained from the Cayley graph of $G$ by ``coning off'' each $P$-coset is $\delta$-hyperbolic.  In analogy, we define $G$ to be ``weakly free'' or \emph{quasi-arboreal} relative to subgroups $\{P\}$ if the coned-off Cayley graph is quasi-isometric to a tree.  We examine this acute form of weak hyperbolicity in Section~\ref{sec:weakrel}, where we obtain the following consequence of Theorem~\ref{thm:HGisQT}.

\begin{corB}\label{cor:corB}
Let $G$ act properly and cocompactly on the CAT(0) cube complex $\widetilde X$.  Then $G$ is quasi-arboreal relative to the set of hyperplane stabilizers.
\end{corB}

Section~\ref{sec:asdim} discusses the asymptotic dimension of cubulated groups.  Recently, in~\cite{Wright2010}, Wright proved a beautiful theorem stating that the asymptotic dimension of a CAT(0) cube complex is bounded above by its dimension, and observed that this implies that groups acting properly on CAT(0) cube complexes have finite asymptotic dimension.  On the other hand, Bell and Dranishnikov~\cite{BellDranishnikov01} showed that a finite graph of asymptotically finite-dimensional groups has finite asymptotic dimension.  Using Wright's theorem on asymptotic dimension of cube complexes and the \emph{Hurewicz-type theorem} of Bell and Dranishnikov~\cite{BellDranishnikov06}, we obtain the following improved statement.

\begin{corC}\label{cor:corC}
Let $G$ be a finitely generated group acting on the locally finite CAT(0) cube complex $\widetilde X$, with $\dimension\widetilde X=D<\infty$.  Suppose there exists $n\in\naturals$ such that for each 0-cube $x$, the stabilizer $G_x$ satisfies $\asdim G_x\leq n$.  Then $\asdim G\leq n+D$.
\end{corC}

Section~\ref{sec:planarbipartiteII} discusses the relationship between Gromov-hyperbolicity of CAT(0) cube complexes and complete bipartite subgraphs of the associated crossing graph.  The primary aim is:

\begin{thmE}\label{thm:thmE}
Let $G$ be a group acting properly and cocompactly on the CAT(0) cube complex $\widetilde X$.  Then exactly one of the following holds:
\begin{enumerate}
\item $G$ is word-hyperbolic.
\item The crossing graph $\Delta$ of $\widetilde X$ contains a complete bipartite graph $K_{\infty,\infty}$.
\end{enumerate}
\end{thmE}

Theorem~\ref{thm:flatplaneonlyif} is a cubical version of the flat plane theorem~(see e.g.~\cite{BridsonHaefliger}).  The theorem is proved by constructing the $K_{\infty,\infty}$ from a sequence of arbitrarily large finite complete bipartite graphs, much as one constructs a plane as a limit of arbitrarily large discs in the proof of the flat plane theorem.

Sections~\ref{sec:prelim} and~\ref{sec:sphere} contain preliminary material: Section~\ref{sec:prelim} summarizes the relevant properties of CAT(0) cube complexes and surveys techniques for manipulating disc diagrams in nonpositively curved cube complexes.  These techniques appear in unpublished lecture notes of Casson, although not strictly in the context of CAT(0) cube complexes.  They were developed further by Sageev in his thesis, and are described extensively by Wise in recent work.  Moreover, Chepoi has used disc diagram techniques in his proof that CAT(0) cube complexes are median spaces~\cite{ChepoiMedian}.  Section~\ref{sec:sphere} describes spheres in contact graphs.\\

\textbf{Acknowledgements:}  I am grateful to Michah Sageev and Dani Wise for helpful discussions, and to Victor Chepoi and Piotr Przytycki for useful criticism.  I also thank anonymous referees for several corrections, clarifications, and simplifications.

\section{Preliminaries}\label{sec:prelim}
The following notions and notations are used throughout.

\subsection{CAT(0) cube complexes}
\begin{defn}[Cube complex]\label{defn:CAT(0)cubecomplex}
For $0\leq n<\infty$, an $n$-\emph{cube} is a copy of the Euclidean cube $\left[-\frac{1}{2},+\frac{1}{2}\right]^n$.  A $d$-\emph{dimensional face} of the $n$-cube $c$ is a subspace obtained by restricting $n-d$ coordinates to $\pm \frac{1}{2}$.
A \emph{cube complex} $X$ is a CW-complex whose $n$-dimensional cells are $n$-cubes, such that the attaching map of each cube $c$ restricts to a combinatorial isometry on each face of $c$, mapping the face to a cube of $X$.

The \emph{link} of a 0-cube $v$ in a cube complex $X$ is the complex made of simplices whose $n$-simplices correspond to the $(n+1)$-cubes that have a corner at $v$, with simplices attached along their faces according to the attaching of the corresponding cubes.

A simplicial complex $S$ is a \emph{flag complex} if each family of $n+1$ pairwise-adjacent 0-simplices in $S$ spans an $n$-simplex, for each $n\geq 0$.  A cube complex $X$ is \emph{nonpositively curved} if the link of $v$ is a flag complex for every 0-cube $v$ of $X$.  A simply-connected nonpositively curved cube complex $\widetilde X$ is called a \emph{CAT(0) cube complex}.
\end{defn}

The term ``CAT(0) cube complex'' is an artifact of the result of Gromov stating that a simply-connected finite-dimensional cube complex satisfying the nonpositive curvature condition of Definition~\ref{defn:CAT(0)cubecomplex} admits a piecewise-Euclidean CAT(0) metric~\cite{Gromov87}.  This also follows from more general results of Bridson, in the finite-dimensional case~\cite{BridsonThesis}, and was extended by Leary to infinite-dimensional cube complexes~\cite{LearyInfiniteCubes}.  When we mention the CAT(0) metric on a CAT(0) cube complex $\widetilde X$, we are referring to this metric.  However, as discussed below, we shall usually use the more natural combinatorial metric on $\widetilde X^{(1)}$.

\begin{defn}[Hyperplane]\label{defn:Hyperplane}
A \emph{midcube} of an $n$-cube $c$ is an $(n-1)$-cube in $c$ obtained by restricting exactly one coordinate to 0.  A \emph{hyperplane} $W$ in the CAT(0) cube complex $\widetilde X$ is a connected union of midcubes of cubes in $\widetilde X$ such that, for each finite-dimensional cube $c$ of $\widetilde X$, either $W\cap c=\emptyset$ or $W\cap c$ is a single midcube of $c$.  The \emph{carrier} $N(W)$ of $W$ is the union of all closed cubes $c$ such that $W$ intersects $c$ in a midcube.

Let $X$ be a nonpositively-curved cube complex, so that the universal cover $\widetilde X$ of $X$ is a CAT(0) cube complex.  An \emph{immersed hyperplane} $\overline W$ of $X$ is the image of a hyperplane $W$ of $\widetilde X$ under the universal covering projection, and the \emph{immersed carrier} $N(\overline W)$ of $W$ is the image of $N(W)$.
\end{defn}

In~\cite{Sageev95}, Sageev proved:

\begin{thm}[Hyperplane properties]\label{thm:HyperplaneProperties}
If $W$ is a hyperplane of the CAT(0) cube complex $\widetilde X$, then:

\begin{enumerate}
  \item $W$ is two-sided, i.e. $N(W)\cong W\times [-\frac{1}{2},\frac{1}{2}]$.
  \item $W$ is \emph{separating}, i.e. $\widetilde X- W$ has exactly two components, called \emph{halfspaces associated to $W$}.
  \item Any midcube is contained in a unique hyperplane.
  \item $W$ is a CAT(0) cube complex whose hyperplanes are of the form $V\cap W$, where $V\neq W$ is a hyperplane of $\widetilde X$ that crosses $W$.
\end{enumerate}
\end{thm}

If $A,B\subset\widetilde X$ are subspaces of the CAT(0) cube complex $\widetilde X$ and $W$ is a hyperplane of $\widetilde X$, then $W$ \emph{separates} $A$ and $B$ if $A$ and $B$ lie in distinct halfspaces associated to $W$.

\begin{defn}[Contacting hyperplanes]\label{defn:hyperplaneadjacency}
Let $\widetilde X$ be a CAT(0) cube complex and $V$ and $W$ a pair of distinct hyperplanes.  A 1-cube $c$ is \emph{dual} to $W$ if the 0-cubes of $c$ are separated by $W$.  Equivalently, $c$ is dual to $W$ if $W$ contains the midcube of $c$.

The hyperplanes $V$ and $W$ \emph{cross} if there is a 2-cube $s$ whose 2 distinct midcubes are contained in $V$ and $W$ respectively.  This is denoted by $V\bot W$.  The hyperplanes $V$ and $W$ \emph{osculate} if they do not cross and there exist distinct 1-cubes $c$ and $c'$, dual to $V$ and $W$ respectively, such that $c$ and $c'$ have a common 0-cube.  In other words, $V$ and $W$ osculate if $N(V)\cap N(W)\neq\emptyset$ and $V$ and $W$ do not cross.

If $V$ and $W$ either cross or osculate, then they \emph{contact}, denoted $V\coll W$.  Note that $V\coll W$ if and only if no hyperplane $U$ separates $V$ from $W$.
\end{defn}

The \emph{dimension} of the CAT(0) cube complex $\widetilde X$ is at least $d$ if $\widetilde X$ contains a $d$-cube.  If $\widetilde X$ contains a $d$-cube but does not contain a $(d+1)$-cube, then $\dimension\widetilde X=d$.  Equivalently, $\dimension\widetilde X$ is equal to $\sup_S|S|$, where $S$ varies over all sets of pairwise-crossing hyperplanes.  The \emph{degree} of $\widetilde X$ is at least $d$ if there exists a 0-cube in $\widetilde X$ with at least $d$ distinct incident 1-cubes.  Equivalently, the degree of $\widetilde X$ is at least $d$ if there is a family of $d$ pairwise-contacting hyperplanes.  Hence the degree of $\widetilde X$ is bounded below by the dimension.

\subsection{Metric notions}
Let $\widetilde X$ be a CAT(0) cube complex and let $\mathcal W$ be the set of hyperplanes of $\widetilde X$.  Consider the standard path-metric $d_{\widetilde X}$ on the graph $\widetilde X^{(1)}$.  It is shown in~\cite{ChepoiMedian} that $\widetilde X^{(1)}$ is a \emph{median graph}: for any three distinct 0-cubes $x,y,z$, there exists a unique 0-cube $m=m(x,y,z)$ such that \[d_{\widetilde X}(x,y)=d_{\widetilde X}(y,m)+d_{\widetilde X}(m,x),\]
\[d_{\widetilde X}(z,y)=d_{\widetilde X}(y,m)+d_{\widetilde X}(m,z),\]
and
\[d_{\widetilde X}(x,z)=d_{\widetilde X}(z,m)+d_{\widetilde X}(m,x).\]

From this characterization, or from Theorem~\ref{thm:HyperplaneProperties}.(2), it follows that a path $P\rightarrow\widetilde X^{(1)}$ is a geodesic if and only if $P$ contains at most one 1-cube dual to each hyperplane of $\widetilde X$.  In other words, $d_{\widetilde X}(x,y)$ counts the number of hyperplanes $W$ such that the 0-cubes $x$ and $y$ lie in distinct halfspaces associated to $W$.

In this paper, all of our arguments are combinatorial, and we shall work with the metric $d_{\widetilde X}$ on the median graph $\widetilde X^{(1)}$.  Accordingly, we adopt the following terminology: unless stated otherwise, all paths in $\widetilde X$ are combinatorial, i.e. a path $P$ in $\widetilde X$ is a continuous combinatorial map $P:I\rightarrow\widetilde X^{(1)}$, where $I$ is a CAT(0) cube complex homeomorphic to an interval.  The path $P$ is a \emph{geodesic} if it is a geodesic path in $\widetilde X^{(1)}$ or, equivalently, if the induced map from the set of hyperplanes of $I$ to the set of hyperplanes of $\widetilde X$ is injective, i.e. if $P$ crosses each hyperplane of $\widetilde X$ at most once.  The subcomplex $Y\subset\widetilde X$ is \emph{isometrically embedded} (\emph{convex}, \emph{bounded}, etc.) if $Y^{(1)}$ is isometrically embedded (convex, bounded, etc.) in $\widetilde X^{(1)}$, with respect to $d_{\widetilde X}$.  By, for example, verifying that its 1-skeleton is \emph{gated}, one sees that for each hyperplane $H$, the carrier $N(H)$ is convex in this sense~\cite{ChepoiMedian}.

As mentioned above, there is a piecewise-Euclidean CAT(0) metric on $\widetilde X$.  In Section~\ref{sec:asdim} and Section~\ref{sec:planarbipartiteII}, we make several statements about the CAT(0) metric, assuming that $\widetilde X$ is finite-dimensional.  This is justified by the following fact: the space $\widetilde X$, with its CAT(0) metric, is quasi-isometric to $\widetilde X^{(1)}$ with the metric $d_{\widetilde X}$ when $\dimension\widetilde X<\infty$.  This fact was proved in greater generality by Bridson~\cite{BridsonThesis}, and a simpler proof in the cubical context appears in~\cite{CapraceSageev}.  The statements about the CAT(0) metric deal with Gromov-hyperbolicity and finite asymptotic dimension, both of which are quasi-isometry invariant properties.

We emphasize, however, that, unless stated otherwise, if we refer to a cubical map $\widetilde Y\rightarrow\widetilde X$ of CAT(0) cube complexes as an isometric embedding, we mean that the image of $\widetilde Y$ is a subcomplex whose 1-skeleton is isometrically embedded in $\widetilde X^{(1)}$ with respect to the combinatorial path-metric $d_{\widetilde X}$.  It is worth noting that Sageev showed that each hyperplane and carrier is convex with respect to the CAT(0) metric~\cite{Sageev95}, and these notions of convexity coincide for full subcomplexes~\cite{HaglundSemisimple}, although we shall not use this fact, or the very natural extension of $d_{\widetilde X}$ to the whole complex considered by Haglund.

\begin{rem}[Cubulated groups]
If a group $G$ acts on a CAT(0) cube complex $\widetilde X$ by cubical automorphisms if $G$ acts on $\widetilde X^{(0)}$, then $G$ stabilizes the set of hyperplanes.  In this situation, $G$ acts on the metric space $(\widetilde X^{(1)},d_{\widetilde X})$ by isometries.

The action is \emph{metrically proper} if, for each bounded subcomplex $B\subset\widetilde X$, there are finitely many $g\in G$ such that $gB\cap B\neq\emptyset$.  When $\widetilde X$ is locally finite, its 1-skeleton is a proper metric space and thus a proper action (in the sense that cube stabilizers are finite) is metrically proper.  Throughout this paper a \emph{cubulated group} is one admitting a metrically proper action by cubical automorphisms on a CAT(0) cube complex.
\end{rem}

\subsection{Cubulating wallspaces}

\begin{defn}\label{defn:wallspace}
A \emph{wallspace} is a pair $(\mathcal S,\mathcal W)$, with $\mathcal S$ a (nonempty) set and $\mathcal W$ a set of \emph{walls}, which are partitions $W$ of $\mathcal S$ into disjoint nonempty \emph{halfspaces} $W^{\pm}$.  Moreover, we suppose that for each $s_1,s_2\in\mathcal S$, there is a finite, nonzero number of walls $W$ that \emph{separate} $s_1$ and $s_2$, in the sense that $s_1$ and $s_2$ lie in distinct halfspaces associated to $W$.

More generally, $W$ \emph{separates} the subsets $A,B\subset\mathcal S$ if $A$ and $B$ lie in distinct halfspaces of $W$, and $W$ separates the walls $U,V$ if it separates some halfspace of $U$ from some halfspace of $V$.

Walls $V,W\in\mathcal W$ \emph{cross} if each of the four \emph{quarterspaces} $V^{\pm}\cap W^{\pm}$ is nonempty.
\end{defn}

\begin{rem}[Sageev's construction]\label{rem:SageevsConstruction}
A wallspace $(\mathcal S,\mathcal W)$ determines a CAT(0) cube complex $\widetilde X$ in such a way that the hyperplanes of $\widetilde X$ correspond to the walls $\mathcal W$ and hyperplanes cross if and only if the corresponding walls do.

An \emph{orientation} of $W$ is a choice of exactly one of the halfspaces associated to $W$, and for each $s\in\mathcal S$, to \emph{orient $W$ towards $s$} is to choose the orientation of $W$ that contains $s$.  More generally, for any subset of $\mathcal S$ that lies in a single halfspace associated to $W$, we speak of orienting $W$ towards that subset.

A \emph{0-cube} is a map $f:\mathcal W\rightarrow\left\{W^{\pm}\mid W\in\mathcal W\right\}$ with the following properties:

\begin{enumerate}
\item (Orientation) For each $W\in\mathcal W$, we have $f(W)\in\{W^-,W^+\}$, i.e. $f$ orients each wall.
\item (Consistency) For all $V,W\in\mathcal W$, we have $f(V)\cap f(W)\neq\emptyset$.
\end{enumerate}

The consistency condition is automatically satisfied for crossing pairs of walls and says that a 0-cube never orients a wall ``away'' from another wall.  The 0-cube $f$ is \emph{canonical} if there exists $s\in\mathcal S$ such that $f(W)$ contains $s$ for each $W\in\mathcal W$.

Denote by $C_0$ the set of all 0-cubes.  The 0-cubes $f_1,f_2\in C_0$ are joined by a 1-cube if and only if there is exactly one wall $W$ such that $f_1(W)\neq f_2(W)$.  We thus obtain a graph $C_1$ whose vertices are the 0-cubes and whose edges are the 1-cubes.  In general, $C_1$ is disconnected, and the cube complex $\widetilde X$ \emph{dual} to the wallspace $(\mathcal S,\mathcal W)$ is constructed from $C_1$ as follows.

Choose any canonical 0-cube $f_s$, which orients each wall toward the element $s\in\mathcal S$.  If $f_t$ is another canonical 0-cube, then since any two points are separated by finitely many walls, $f_s$ and $f_t$ differ on finitely many walls, and thus belong to the same component of $C_1$.  Denote by $\widetilde X^{(1)}$ this \emph{canonical component}.  One then verifies that $\widetilde X^{(1)}$ is the 1-skeleton of a uniquely determined CAT(0) cube complex $\widetilde X$, which is independent of the choice of canonical 0-cube.  $\widetilde X$ is the cube complex \emph{dual} to the wallspace $(\mathcal S,\mathcal W)$ and is completely determined by that data.  The set of hyperplanes of $\widetilde X$ corresponds bijectively to $\mathcal W$, and two hyperplanes contact if and only if the corresponding walls are not separated by a third wall.  Two hyperplanes cross if and only if the corresponding walls cross.

In general, the non-canonical components of $C_1$ are 1-skeleta of CAT(0) cube complexes constructed from ``cubes at infinity''; their 0-cubes are consistent orientations of all walls that differ on infinitely many walls from any canonical 0-cube.

The above construction, when $\mathcal S$ is a finitely-generated group and the walls arise from codimension-1 subgroups, is due to Sageev~\cite{Sageev95}.  The general notion of a wallspace was first introduced in~\cite{HaglundPaulin98}.  Discussions of Sageev's construction in a general wallspace setting appear in~\cite{ChatterjiNiblo04},~\cite{NicaCubulating04} and~\cite{HruskaWise}.

Sageev's construction is sometimes given in terms of principal ultrafilters on the wallspace $(\mathcal S,\mathcal W)$.  We use the notation $W^+$ and $W^-$ for the halfspaces associated to the wall $W$.  In the following definition, these are merely notations; we are not, in the following definition, designating a map choosing a halfspace for each wall.

\begin{defn}\label{defn:ultrafilter}
An \emph{ultrafilter} on the wallspace $\left(\mathcal S,\mathcal W\right)$ is a set $\omega$ of halfspaces associated to walls in $\mathcal W$ subject to the following conditions:
\begin{enumerate}
\item For all walls $W$, exactly one of the following occurs: $W^+\in\omega$ or $W^-\in\omega$.
\item For any pair $W^+\subset V^+$ of nested halfspaces such that $W^+\in\omega$, we have $V^{+}\in\omega$, and likewise for the other halfspaces associated to $V,W$.
\end{enumerate}
For $s\in\mathcal S$, the \emph{principal ultrafilter $\omega_s$} associated to $s$ is the set of all halfspaces containing $s$.
\end{defn}

Definition~\ref{defn:ultrafilter} gives an equivalent construction of the cube complex dual to a wallspace.  First, note that any ultrafilter $\omega$ on $(\mathcal S,\mathcal W)$ corresponds to a 0-cube of $C_1$: the inclusion in $\omega$ of exactly one halfspace associated to each wall orients all of the walls.  The second condition in Definition~\ref{defn:ultrafilter} is a paraphrase of the consistency condition on orientations of the set of walls.  It is easily seen that the principal ultrafilter $\omega_s$ corresponds to the 0-cube that orients each wall toward the element $s\in\mathcal S$.  Moreover, the 0-cubes corresponding to the ultrafilters $\omega_1,\omega_2$ belong to the same component of the graph $C_1$ if and only if the symmetric difference $\omega_1\triangle\omega_2$ is finite.  Therefore, since any two elements of $\mathcal S$ are separated by finitely many walls, any two principal ultrafilters have finite symmetric difference and thus the corresponding 0-cubes belong to the same component of $C_1$.
\end{rem}

We will apply Sageev's construction later to establish a few statements about crossing graphs and contact graphs.

\subsection{Disc diagrams in CAT(0) cube complexes}
This subsection summarizes parts of the discussion of disc diagrams in CAT(0) cube complexes appearing in \cite{WiseIsraelHierarchy}.

\begin{defn}\label{defn:SquareDiagram}
Let $X$ be a nonpositively curved cube complex.  A \emph{disc diagram $D\rightarrow X$ in $X$} is a continuous combinatorial map of cube complexes, where $D$ is a \emph{disc diagram}: a contractible, finite, 2-dimensional cube complex equipped with a fixed (topological) embedding into $S^2$.  The \emph{area} of $D$ is the number of 2-cubes in $D$.

Since $D$ is contractible, the complement of $D$ in $S^2$ is a 2-cell whose attaching map is the \emph{boundary path} $\partial_pD$ of $D$.  If $D\rightarrow X$ is a disc diagram in $X$, then the restriction of this map to the boundary path of $D$ is a combinatorial path $\partial_pD\rightarrow X$.  Note that $\partial_pD$ may not be injective on 0-cubes or 1-cubes.  If $X$ is simply-connected, then any closed combinatorial path in $X$ is the boundary path of a disc diagram $D\rightarrow X$.

Fixing an immersed hyperplane $W$ of $X$, consider the set of midcubes in $D$ that map to $W$.  A maximal concatenation of such midcubes is a \emph{dual curve} $C$ in $D$ mapping to $W$.  Note that each dual curve is a singular curve: each 1-cube in $D$ has at most two incident 2-cubes, and thus each 0-cell of $C$ has valence at most 2, though $C$ may cross itself in the interior of one or more 2-cubes.

A 1-cube of $D$ whose midcube is contained in a dual curve $C$ is \emph{dual to $C$}.  An \emph{end} of a dual curve $C$ is a midpoint of a 1-cube of $\partial_pD$ dual to $C$.  The \emph{carrier} of the dual curve $C$ is the union of closed 2-cubes of $D$ that contain midcubes belonging to $C$.

A dual curve $C$ with 0 ends is a \emph{nongon}. If $C$ is not a nongon, then it has two ends.  A \emph{monogon} is a closed subpath of a dual curve that crosses itself in the initial 2-cube of its carrier, which is equal to the terminal 2-cube.  Any dual curve that crosses itself contains a monogon.  An \emph{oscugon} is a closed subpath $C'$ of a dual curve $C$ such that $C'$ does not self-cross, such that the two distinct terminal 1-cubes of $c$ have a common 0-cube but do not form the corner of a 2-cube in $D.$  A \emph{bigon} is a pair of dual curves that cross in two distinct squares of $D$.  See Figure~\ref{fig:immense}.
\begin{figure}[h]
  \includegraphics[width=3in]{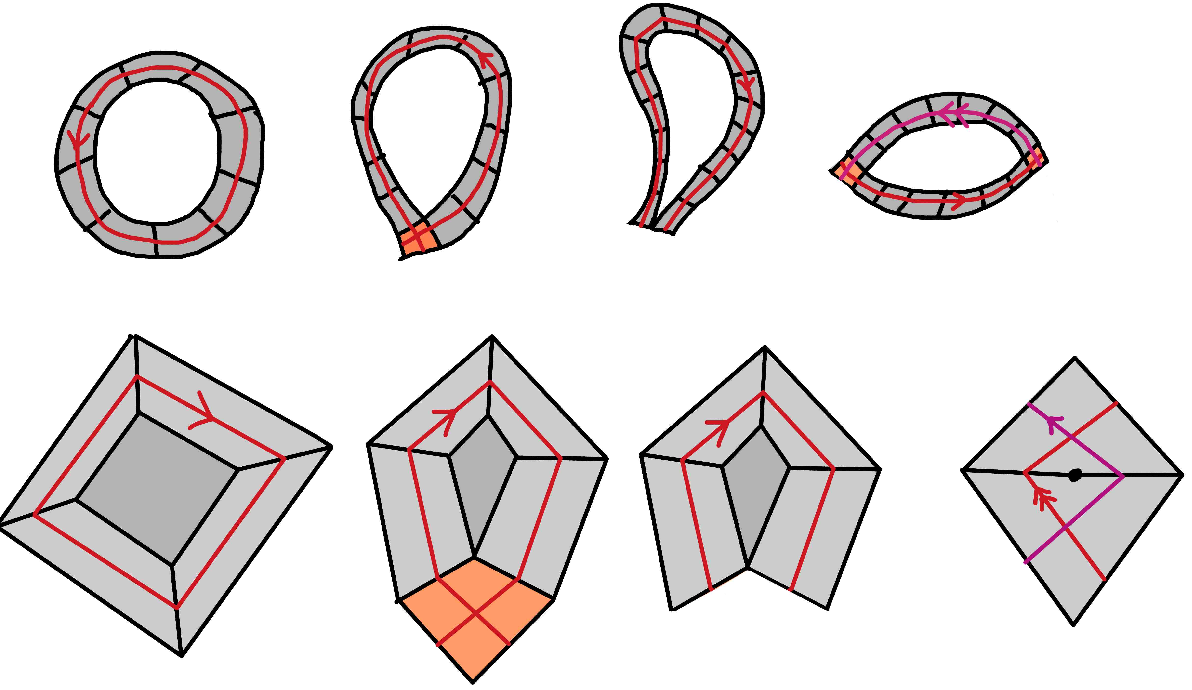}\\
  \caption{Left to right at the top are heuristic pictures of the carriers of a: nongon, monogon, oscugon, and bigon. Below each of these figures is an actual disc diagram containing the corresponding configuration.  In both sets of pictures, the dual curve itself is decorated with an arrow.}\label{fig:immense}
\end{figure}
If $C$ is a dual curve in $D$ whose ends lie on subpaths $P,Q$ of $\partial_pD$, it is often convenient to say that $K$ \emph{emanates} from $P$ and \emph{terminates} on $Q$ (or vice versa), or that $K$ \emph{travels from $P$ to $Q$.}
\end{defn}

\subsection{Complexity reductions in disc diagrams}
The techniques used to prove the following lemma are discussed in detail in~\cite{WiseIsraelHierarchy} and were developed from ideas of Casson (see~\cite{Sageev95}).

\begin{lem}[\cite{WiseIsraelHierarchy}]\label{lem:minareaDD}
Let $P\rightarrow X$ be a closed combinatorial path in a nonpositively curved cube complex $X$ and let $D\rightarrow X$ be a minimal-area disc diagram among all diagrams $D'$ with $\partial_pD'=P$.  Then $D$ contains no nongons, monogons, oscugons, or bigons.
\end{lem}

We refer the reader to~\cite{WiseIsraelHierarchy} for a discussion of \emph{cancellable pairs} and \emph{hexagon moves} in disc diagrams over nonpositively-curved cube complexes, which are used in the proof of Lemma~\ref{lem:minareaDD}.  In Sections~\ref{sec:sphere} and~\ref{sec:quasitree}, we will study a particular type of disc diagram and will apply slightly different techniques than those used in the proof of Lemma~\ref{lem:minareaDD}.  In particular, while a combination of hexagon moves and cancellable pair removals is used in~\cite{WiseIsraelHierarchy} to modify a disc diagram without affecting its boundary path, we will sometimes make certain changes to the boundary path, as follows.

Let $H_0\coll H_1\coll\ldots H_{n-1}\coll H_0$ be (not necessarily pairwise distinct) hyperplanes contacting (at least) as indicated.  Then we can choose, for each $i\in\integers_n$, a combinatorial geodesic $P_i\rightarrow N(H_i)$ so that there is a closed path $P\rightarrow\widetilde X$ that is the concatenation $P=\prod_{i=1}^{n-1}P_i$.  Since $\widetilde X$ is a CAT(0) cube complex, there is a disc diagram $D\rightarrow\widetilde X$ such that $\partial_pD=P$.  This situation is shown schematically in Figure~\ref{fig:fixedcorners}.
\begin{figure}[h]
  \includegraphics[width=2.5in]{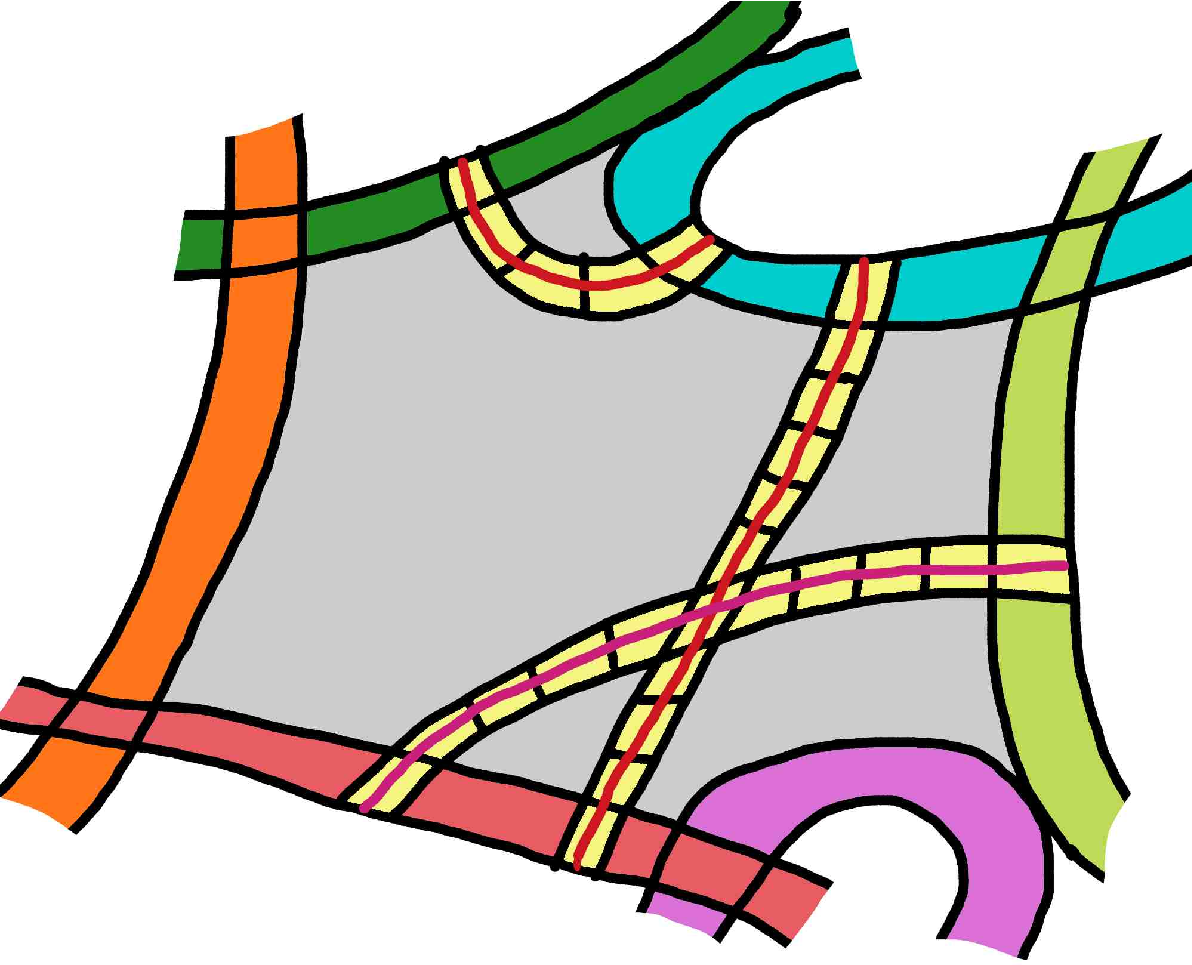}\\
  \caption{This heuristic picture shows The image of a disc diagram whose boundary path is the concatenation of geodesic segments lying on a fixed collection of hyperplane carriers, along with parts of those carriers.  The two configurations of dual curves precluded by Lemma~\ref{lem:fixedcorners} in the minimal-complexity case are shown.}\label{fig:fixedcorners}
\end{figure}

In our applications, the collection $\{H_i\}$ of hyperplanes is fixed.  Given such a collection $\{H_i\}$ of hyperplanes, forming a closed path in the \emph{contact graph} (the intersection graph of the set of hyperplane-carriers), a disc diagram $D$ constructed as above is a \emph{diagram with fixed carriers} for the closed path $\sigma=H_0\coll H_1\coll\ldots H_{n-1}\coll H_0$.  Note that it is possible that $H_i\coll H_j$ for $|i-j|>1$, but this is not necessarily reflected in $D$.

The \emph{complexity} $\comp(D)$ of $D$ is the pair $(\area(D),|P|)$, taken in lexicographic order.  Suppose that $D$ is of minimal complexity among all diagrams bounded by paths $P$ that decompose in the above fashion, i.e. among all diagrams with fixed carriers for $\sigma$.  In particular, $D$ is of minimal area among all disc diagrams with boundary path $P$, so that, by Lemma~\ref{lem:minareaDD}, $D$ does not contain any nongons, monogons, oscugons, or bigons.

Let $K$ be a dual curve in $D$.  Then $K$ has one end on $P_i$ and another end on $P_j$.  Since each of $P_i$ and $P_j$ is a geodesic, we must have $i\neq j$, for otherwise $P$ would contain two distinct 1-cubes dual to the same hyperplane, namely the hyperplane to which $K$ maps.

Suppose that $K$ is a dual curve traveling from $P_i$ to $P_j$, and $K'$ a dual curve traveling from $P_i$ to some $P_k$, such that $K$ and $K'$ cross in a 2-cube $s$ of $D$, as at the bottom of Figure~\ref{fig:fixedcorners}.  Let $P'_i$ be the smallest connected subpath of $P_i$ containing the 1-cubes of $P_i$ dual to $K$ and $K'$.  Let $Q$ and $Q'$ be shortest combinatorial paths in $D$ that start at the ends of $P'_i$, travel along the carriers of $K$ and $K'$ respectively, meeting at the corner of $s$ that is separated from $P_i$ by $K$ and $K'$.  Let $E$ be the subdiagram of $D$ bounded by $P'_i,Q$ and $Q'$.  See Figure~\ref{fig:diagramEF}, at left.  If $C$ is a dual curve in $E$ emanating from $P'_i$, then $C$ cannot end on $P'_i$ by the fact that $P'_i$ is a geodesic, and hence $C$ crosses $K$ or $K'$.  Suppose the former.  Then $K$ and $C$ form a triangle of dual curves in $D$ that is properly contained in $E$.  Hence, by choosing an innermost such triangle, we may assume that $|P'_i|=2$.  Hence $P_i$ contains a path $c_1c_2$, where $c_1$ and $c_2$ are 1-cubes of $N(H_i)$ that form the corner of a 2-cube $s$ in $D$, as at right in Figure~\ref{fig:diagramEF}.  Let $Q_i$ be the path in $\widetilde X$ obtained by removing the subpath $c_1c_2$ from $P_i$ and replacing it by $c'_1c'_2$, where $c'_1$ is the 1-cube of $s$ opposite to $c'_2$ and $c'_2$ the 1-cube of $s$ opposite to $c_1$.
\begin{figure}[h]
  \includegraphics[width=2.5in]{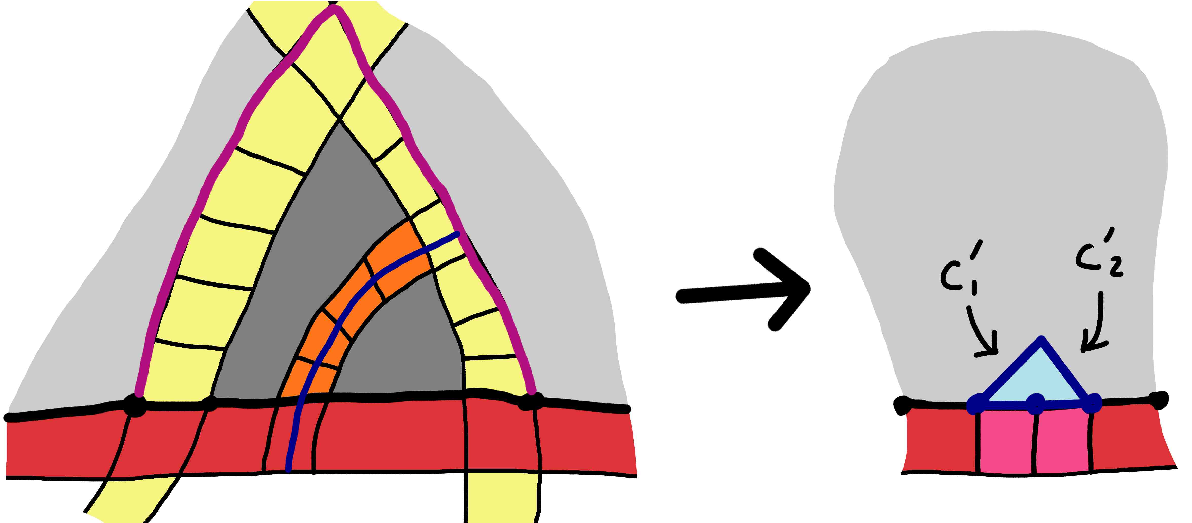}\\
  \caption{The diagram $E$ arising when two dual curves emanating from $P_i$ cross.}\label{fig:diagramEF}
\end{figure}

Note that $|Q_i|=|P_i|$ and $Q_i$ has the same endpoints as $P_i$; in particular, $Q_i$ is a geodesic.  Moreover, $Q_i$ maps to $N(H_i)$.  To see this, it suffices to show that $c'_1$ and $c'_2$ map to $N(H_i)$.  The 1-cubes $c_1$ and $c'_2$ are dual to a hyperplane $W_1$ and $c'_1$ and $c_2$ are dual to a hyperplane $W_2$.  If $W_1=H_i$, then $c'_2$ is dual to $H_i$, whence $s$, and thus $c'_1$, maps to $N(H_i)$.  Hence suppose that $W_1,W_2$, and $H_i$ are all distinct.  Then $W_1,W_2$ cross $H_i$ and $W_1$ crosses $W_2$ in the 2-cube $s$.  By nonpositive curvature, $s$ lies in a 3-cube of $N(W_1)\cap N(W_2)\cap N(H_i)$ and in particular $Q_i\rightarrow N(H_i)$.

By removing $s$ from $D$ and replacing $P_i$ with $Q_i$, we replace $D$ by a proper subdiagram diagram $D'$ that has fixed carriers for $\{H_i\}$, so that $\comp(D')<\comp(D)$, contradicting the minimality of $D$.  Hence no two dual curves in $D$ emanating from any $P_i$ can cross.

Now consider the case in which $K$ emanates from $P_i$ and terminates on $P_{i+1}$.  This is shown at the top of Figure~\ref{fig:fixedcorners} and is enlarged in Figure~\ref{fig:fixedcorners2}.  Let $F$ be the subdiagram of $D$ between and including the carrier of $K$ and the subtended parts of $P_i,P_{i+1}$.  Suppose also that $K$ is innermost, in the sense that no dual curve $L$ in $F$ travels from $P_i$ to $P_{i+1}$ (otherwise, we could argue using $L$ instead of $K$).  If there is a dual curve $K'$ in $F$, distinct from $K$, emanating from $P_i$ or $P_{i+1}$, then $K'$ crosses $K$ in $D,$ contradicting minimality of the complexity, by the preceding argument.  Hence $|K|=0$ and $P_i$ and $P_{i+1}$ have a common 1-cube $c$ dual to the hyperplane to which $K$ maps, as at right in Figure~\ref{fig:fixedcorners2}.
\begin{figure}[h]
  \includegraphics[width=2.5in]{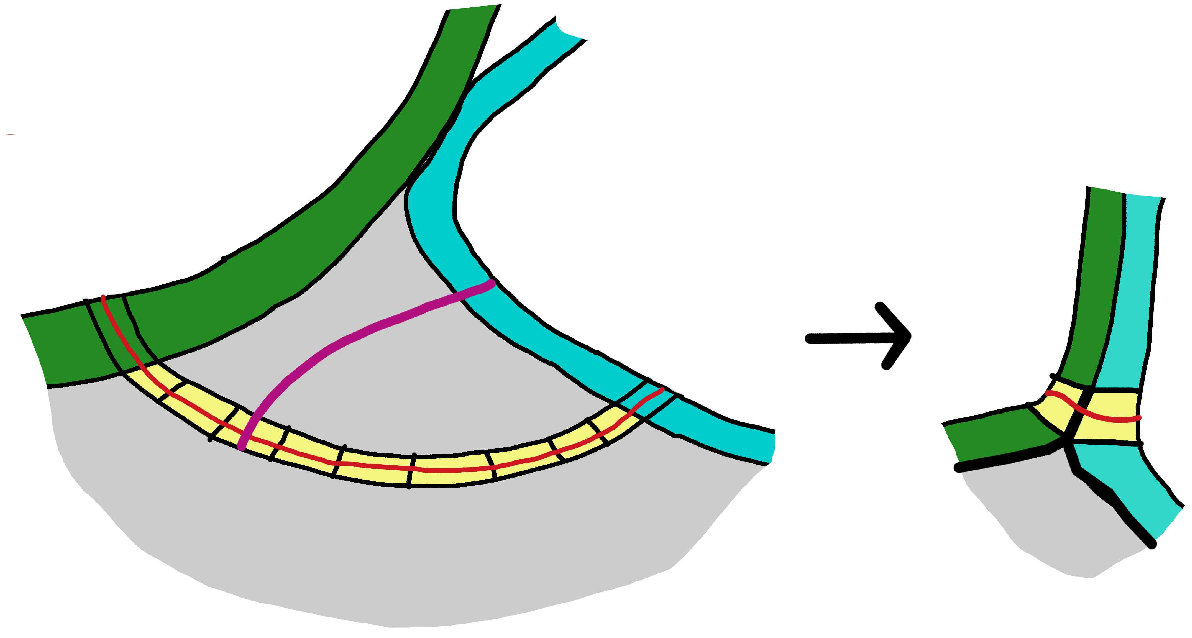}\\
  \caption{Chopping off a spur using the diagram $F$.}\label{fig:fixedcorners2}
\end{figure}
The minimal case is that in which the 1-cube $c$ is a \emph{spur} in the language of~\cite{WiseIsraelHierarchy}, i.e. $\partial_pD$ contains the path $cc^{-1}$. By removing $c$ from $P_i$ and $P_{i+1}$, we obtain a subdiagram $D'$ of $D$ that has fixed carriers for $\sigma$ but that has lower complexity.  If not, then since every dual curve in $F$ travels from $P_i$ to $P_{i+1}$, there is a spur in $F,$ which we find by noting that the terminal 1-cube of $P_i$ coincides with the initial 1-cube of $P_{i+1}$.  In either case, we can remove the part of $P_iP_{i+1}$ between the two occurrences of $c$ and lower the complexity of $D$ while preserving the fixed carriers.  We have thus proved:

\begin{lem}\label{lem:fixedcorners}
Let $\sigma=H_0\coll H_1\coll\ldots\coll H_{n-1}\coll H_0$ be a closed path in the contact graph and let $D$ be a diagram with fixed carriers for $\sigma$.  Suppose that $D$ is of minimal complexity among all diagrams with fixed carriers for $\sigma$, and let $\partial_pD=P_0P_1\ldots P_{n-1}$ be the boundary path, where each $P_i\rightarrow N(H_i)$ is a combinatorial geodesic.  Then:
\begin{enumerate}
\item For all $i\in\integers_n$, no dual curve $K$ emanating from $P_i$ terminates on $P_{i\pm 1}$.
\item If $K,K'$ are dual curves emanating from $P_i$, then $K$ and $K'$ do not cross.
\end{enumerate}
\end{lem}

We emphasize that it is possible for $K$ to emanate from $P_i$ and terminate on the next positive-length labeled subpath of the path $P$.  More precisely, if $|P_{i+1}|=0$, we must still treat $P_{i+1}$ as one of the designated geodesic subpaths of $P$, since $D$ has fixed carriers.  In this case, as above, $P_i$ and $P_{i+2}$ intersect in a spur $c$ mapping to a 1-cube dual to the hyperplane to which $K$ maps, but we cannot remove $c$, since that would remove the path $P_{i+1}$ (which is an endpoint of $c$) and result in a diagram that does not have fixed carriers.  However, in such a situation, we reach a conclusion that $H_i\coll H_{i+2}$.

\subsection{Crossing graphs and contact graphs}
Unless stated otherwise, graphs in this paper are simplicial in the sense that they have no loops or multi-edges.  Also, graphs have the combinatorial metric, with all edges of length 1.

\begin{defn}\label{defn:fullballsetc}
If $\Phi$ is a subgraph of a graph $\Lambda$, then $\full(\Phi)$ denotes the full subgraph of $\Lambda$ generated by the vertices of $\Phi$.

Let $v$ be a vertex of $\Lambda$ and let $n\geq 0$.  The \emph{full ball} $\bar B_n(v)=\full\left(\left\{w\in\Phi^{(0)}\,:\,d(v,w)\leq n\right\}\right)$.

The \emph{full sphere} $\bar S_n(v)$ denotes the full subgraph of $\Lambda$ generated by vertices at distance exactly $n$ from $v$.
\end{defn}

We shall often need the following facts about subcomplexes of cube complexes.  Let $Y\subset\widetilde X$ be an isometrically embedded subcomplex of the CAT(0) cube complex $\widetilde X$.  Then for each hyperplane $H$ of $\widetilde X$, either $H\cap Y=\emptyset$, or $H\cap Y$ is a connected subspace of $Y$ such that $Y-(H\cap Y)$ has two components, one in each halfspace associated to $H$.  In the latter case, $H$ \emph{crosses} $Y.$  Conversely, if $Y\subset\widetilde X$ is connected and has the property that $H\cap Y$ is connected (or empty) for each hyperplane $H$, then $Y\hookrightarrow\widetilde X$ is an isometric embedding.

Indeed, if $Y\cap H$ is connected for each $H$, let $a,b\in Y$ be 0-cubes.  Let $P$ be a geodesic of $Y$ joining $a,b$.  Suppose that $P=P_1c_1P_2c_2P_3$, where $c_1,c_2$ are 1-cubes dual to the same hyperplane $H$.  Then there is a sequence $c_1=d_0,d_1,d_2,\ldots,d_k,d_{k+1}=c_2$ of 1-cubes, all dual to $H$, such that $d_i$ and $d_{i-1}$ lie on opposite sides of the same 2-cube of $N(H)$ and the midcube of each $d_i$ lies in $H\cap Y$.  Since $Y$ is a subcomplex, each $d_i\subset Y$, and thus there is a geodesic $Q\rightarrow N(H)\cap Y$ joining the initial 0-cube of $c_1$ to the terminal 0-cube of $c_2$, such that $Q$ does not cross $H$.  On the other hand, $Qc_1P_2c_2$ bounds a minimal disc diagram in which all dual curves emanating from $Q$ end on $P_2$, and thus $|Q|\leq|P_2|$.  But then $P_1QP_3$ is a path in $Y$ joining $a,b$ with $|P_1QP_3|<|P|$, a contradiction.  Hence every 1-cube of $P$ is dual to a distinct hyperplane, whence $P$ is a geodesic of $\widetilde X$ and thus $Y$ is isometrically embedded.

Lemma~\ref{lem:convexsubwallspace} says that a locally convex subcomplex of a CAT(0) cube complex is convex, and can be proved using disc diagrams.

\begin{lem}\label{lem:convexsubwallspace}
Let $\widetilde X$ be a CAT(0) cube complex with a set $\mathcal W$ of hyperplanes.  Let $C\subseteq\widetilde X$ be a connected subcomplex and let $\mathcal W'$ be the set of hyperplanes of $\widetilde X$ that cross $C$.  Then the following are equivalent:
\begin{enumerate}
  \item The subcomplex $C\subset\widetilde X$ is convex.
  \item Distinct hyperplanes $V,W\in\mathcal W'$ cross in $\widetilde X$ if and only if they cross in $C$, i.e. for some 2-cube $s$ dual to $V$ and $W$, we have $s\subset C$.
\end{enumerate}
\end{lem}

Lemma~\ref{lem:convexsubwallspace} can be proved using minimal-area disc diagrams and hexagon moves; see~\cite{WiseIsraelHierarchy}.

\begin{rem}\label{rem:contactingandconvexity}
More generally, an isometrically embedded subcomplex $C$ is convex if and only if any pair of hyperplanes $V,W$ of $\widetilde X$ such that $N(V)\cap C$ and $N(W)\cap C$ are both nonempty contact if and only $N(V)\cap N(W)\cap C\neq\emptyset$.  This is a special case of Helly's theorem for CAT(0) cube complexes, which is stated below and which is discussed in, for example,~\cite{RollerHabilitation}.  Helly's theorem appears in many different contexts.  For example, convexity and the Helly property are discussed in the context of median spaces in van de Vel's book~\cite{vanderVelBook}.

Chepoi has also pointed out in private communication that Lemma~\ref{lem:HellysTheorem} also follows from the median property of $\widetilde X^{(1)}$.  Indeed, convex subsets of a median graph are \emph{gated}, collections of gated subsets enjoy the Helly property.
\end{rem}

\begin{lem}\label{lem:HellysTheorem}
Let $\widetilde X$ be a CAT(0) cube complex and let $Y_1,Y_2,\ldots, Y_n$ be a finite collection of convex subcomplexes of $\widetilde X$.  Suppose that $Y_i\cap Y_j\neq\emptyset$ for all $1\leq i\leq j\leq n$.  Then $\bigcap_iY_i\neq\emptyset$.
\end{lem}

\begin{defn}[Contact graph, crossing graph]\label{defn:hyperplaneandwallgraphs}
Let $\widetilde X$ be a CAT(0) cube complex.  The \emph{contact graph} $\Gamma$ of $\widetilde X$ is the graph whose vertices are the hyperplanes of $\widetilde X$, with hyperplanes $V$ and $W$ joined by an edge if and only if $V\coll W$.  Equivalently, $\Gamma$ is the nerve of the covering of $\widetilde X$ by the set of hyperplane carriers.  The \emph{crossing graph} $\Delta$ of $\widetilde X$ is the subgraph of $\Gamma$ containing all of the vertices, with $V$ and $W$ joined by an edge exactly when $V\bot W$.  When discussing these graphs, the terms ``vertex'' and ``hyperplane'' are used interchangeably.

While $\Gamma$ is always connected, $\Delta$ may not be, as in the following example.
\end{defn}

\begin{exmp}\label{exmp:hwgraphexmp}
When $\widetilde X$ is a tree, the hyperplanes are the midcubes of the edges.  The vertices of $\Gamma$ correspond to the 1-cubes of $\widetilde X$, with two vertices adjacent exactly when the corresponding 1-cubes have a common 0-cube.  In particular, for each 0-cube of $\widetilde X$, the contact graph contains a complete graph whose vertex-set has cardinality equal to the valence of that 0-cube.  The crossing graph $\Delta$ has no edges.

If $\widetilde X$ is a cube, $\Delta=\Gamma$ is the complete graph on the set of midcubes.

If $\widetilde X$ is the standard tiling of $\reals^n$ by $n$-cubes, then $\Delta$ is a complete $n$-partite graph with each class of the $n$-partition an order-isomorphic copy of $\integers$, where the hyperplanes in each class are ordered by designating a halfspace for each in such a way that the designated halfspaces are totally ordered by inclusion.  Adding to $\Delta$ an edge between consecutive vertices in each class gives the contact graph $\Gamma$.  More generally, if $\widetilde X$ and $\widetilde Y$ are CAT(0) cube complexes, then the contact graph of $\widetilde X\times\widetilde Y$ is the join of the contact graphs of the factors.
\end{exmp}

The following basic notion is useful in Section~\ref{sec:planarbipartiteII}.

\begin{defn}[Inseparable set of hyperplanes]\label{defn:inseparable}
Let $\mathcal W$ be the set of hyperplanes in the CAT(0) cube complex $\widetilde X$.  The set $\mathcal W'\subseteq\mathcal W$ is \emph{inseparable} if, for any two $W_1,W_2\in\mathcal W'$, no hyperplane $W_3\in\mathcal W-\mathcal W'$ separates $W_1$ from $W_2$.
\end{defn}

The following proposition shows that the class of graphs that are crossing graphs of CAT(0) cube complexes is very large.

\begin{prop}\label{prop:cubulatinggraphs}
For any simplicial graph $\Delta$, there exists a CAT(0) cube complex $\widetilde X$ whose crossing graph is $\Delta$.
\end{prop}

\begin{proof}
First suppose that $\Delta$ is connected and does not consist of a single vertex.  We first construct a wallspace from $\Delta.$

For each $v\in\Delta^{(0)}$, let $I(v)$ be a set of vertices with the same cardinality as the set of vertices of $\Delta$ adjacent to $v$, together with two additional vertices $a(v),b(v)$.  There is an \emph{augmented graph} $\Delta^{\sharp}$ formed by inflating each vertex of $\Delta$ into a disjoint set of vertices according to the valence of $v$.  More precisely, $\Delta^{\sharp}$ is the graph whose vertices are $\coprod_{v\in\Delta^{(0)}}I(v)$, and whose edges are as follows.  If $v$ and $w$ are adjacent vertices of $\Delta$, join some vertex of $I(v)$ to some vertex of $I(w)$ by an edge, and do this in such a way that the resulting graph has the property that all vertices in each $I(v)-\{a(v),b(v)\}$ have exactly one incident edge; the remaining vertices have valence 0.  Write $e\sim f$ when $e$ and $f$ are adjacent in $\Delta^{\sharp}$.

The underlying set of the wallspace is $S=\left(\Delta^{\sharp}\right)^{(0)}$.  For each $w\in\Delta^{(0)}$, define a wall $(w^+,w^-)$ by:
\begin{eqnarray*}
  w^+ &=& (I(w)-\{b(w)\})\cup\left\{f\,:\,\exists e\in I(w),\,e\sim f\right\} \\
  w^- &=& S- w^+.
\end{eqnarray*}
By construction, two walls in $\mathcal W$ cross if and only if the corresponding vertices of $\Delta$ are adjacent.  Indeed, let $v$ and $w$ be adjacent vertices of $\Delta$.  Then $w^+\cap v^+$ contains the vertices of $I(v)\cup I(w)$ corresponding to the endpoints of the edge of $\Delta$ joining $v$ and $w$.  The intersection of $v^+$ and $w^-$ consists of the elements of $I(w)$ that do not correspond to the edge joining $v$ to $w$.  The extra element $a(w)$ of $I(w)$ guarantees that there is at least one of these.  On the other hand, if $v$ and $w$ are non-adjacent, then $(I(w)-b(w))\cap\left\{f\,:\,\exists e\in I(v),\,e\sim f\right\}=\emptyset$, so $v^+\cap w^+=\emptyset$, since $I(v)\cap I(w)=\emptyset$ for all $v,w$.  Finally, $w^-$ contains $b(w)$, by definition.  On the other hand, $b(w)\not\in I(v)$, since $v\neq w$, and $b(w)$ is not adjacent to any vertex in $I(v)$, so that $b(w)\in v^-$.  Hence $w^-\cap v^-\neq \emptyset.$

The cube complex $\widetilde X$ dual to $(S,\mathcal W)$ therefore has crossing graph $\Delta$.  Indeed, two hyperplanes of $\widetilde X$ cross if and only if the corresponding walls cross.  If $\Delta$ consists of a single vertex, then define $\widetilde X$ to be a single 1-cube.

If $\Delta_1$ and $\Delta_2$ are distinct components of $\Delta$, then the preceding construction can be performed independently on each component that has more than one vertex, and the resulting cube complexes attached along a single 0-cube, adding osculations, but not crossings, of hyperplanes.  In fact, every CAT(0) cube complex with disconnected crossing graph consists of a collection of cube complexes with connected crossing graphs, glued along various 0-cubes.
\end{proof}

The proof of Proposition~\ref{prop:cubulatinggraphs} shows that $\Delta$ does not uniquely determine $\widetilde X$ if $\Delta$ is disconnected, but this nonuniqueness can happen in other ways.  For example, consider $\Delta=K_{2,3}$, a complete bipartite graph with 5 vertices.  Then $\Delta$ is the crossing graph of $[-1,1]\times[-2,1]$, and is also the crossing graph of $T\times[-1,1]$, where $T$ is a tripod.  However, these two complexes have different contact graphs; one is the join of two line segments and one is the join of a line segment and a triangle.

\begin{prop}[Recubulation]\label{prop:recube}
Let $\widetilde X$ be a CAT(0) cube complex with contact graph $\Gamma$.  There exists a CAT(0) cube complex $\widetilde X_r$ whose crossing graph is equal to $\Gamma$, and there is an isometric embedding $\widetilde X\rightarrow\widetilde X_r$.
\end{prop}

\begin{proof}
The larger cube complex $\widetilde X_r$ is constructed from $\widetilde X$ by \emph{recubulating}.  For each edge $V\coll W$ of $\Gamma$, there is a set $\{(c,c')\}$ of all pairs of 1-cubes such that $c$ is dual to $V$ and $c'$ is dual to $W$, such that $c$ and $c'$ meet in a 0-cube.  Since hyperplanes in a CAT(0) cube-complex do not self-osculate, each pair $(c,c')$ determines a unique 0-cube $c\cap c'$.  For each such pair $(c,c')$ corresponding to an osculation, attach a square $s$ to $\widetilde X$ by gluing two consecutive edges of $s$ along $cc'$ and let $\widetilde X'$ be the (possibly not nonpositively curved) auxiliary cube complex obtained from $\widetilde X$ by attaching all such squares for all osculation-edges $V\coll W$ in $\Gamma$.  Each hyperplane $W$ of $\widetilde X$ extends to a subspace $W'\subset\widetilde X'$ that separates $(\widetilde X')^{(0)}$ into two disjoint subsets.  Indeed, the midcube of $s$ dual to $c$ is added to $V$ and likewise for $c'$ and $W$.

The correspondence $W\mapsto W'$ is bijective, since the 1-cubes of each new square $s$ are in two distinct parallelism classes corresponding to the original pair of osculating hyperplanes, which were distinct.  Hence $\left(\left(\widetilde X'\right)^{(0)},\{W'\}\cong\Gamma^{(0)}\right)$ is a wallspace with the property that two walls $W'$ and $V'$ cross if and only if the corresponding vertices in $\Gamma$ are adjacent.  Cubulating this wallspace gives the desired CAT(0) cube complex $\widetilde X_r$.  See Figure~\ref{fig:ReCube}.  Note that $\widetilde X$ isometrically embeds in $\widetilde X_r$.
\begin{figure}[h]
  \includegraphics[width=3.25 in]{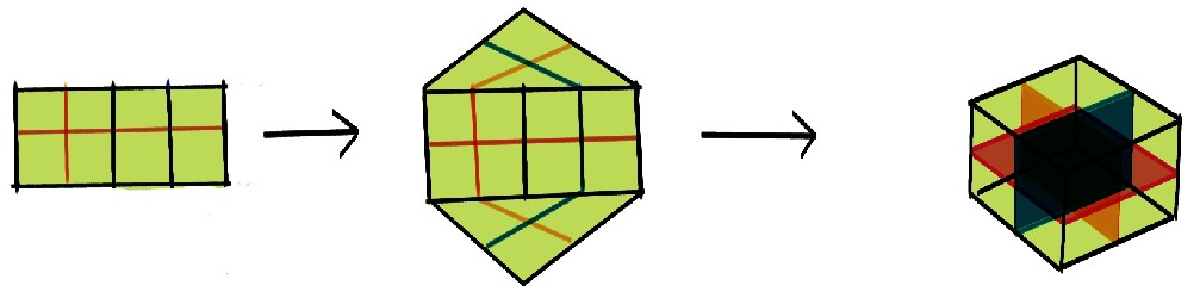}\\
  \caption{Turning osculations into crossings.  The two vertical hyperplanes osculate at left. For each of the two pairs of intersecting 1-cubes corresponding to this osculation, we add a 2-cube, to obtain the middle picture.  To maintain nonpositive curvature, we must now add a 3-cube, and we obtain the picture at right.}\label{fig:ReCube}
\end{figure}
\end{proof}

\begin{rem}
Proposition~\ref{prop:recube} can be proved more topologically by noting that $\widetilde X'$ deformation retracts to $\widetilde X$ and is thus simply connected.  Higher-dimensional cubes can be added to $\widetilde X'$ where necessary to produce the CAT(0) cube complex $\widetilde X_r$ whose hyperplanes correspond to those of $\widetilde X$, since the hyperplanes of a CAT(0) cube complex are determined by the 1-skeleton and each parallelism-class of 1-cubes (explained in e.g.~\cite{ChepoiMedian,HaglundWiseSpecial}) is already represented in $\widetilde X$ and thus in $\widetilde X'$.  Each osculation of hyperplanes in $\widetilde X$ is replaced by a crossing in $\widetilde X_r$ and the dimension of $\widetilde X_r$ is thus equal to the cardinality of the largest clique in $\Gamma$.
\end{rem}

\section{Full spheres in contact graphs}\label{sec:sphere}
Recall that the \emph{full sphere} $\bar S_n(V)\subseteq\Gamma$ is the full subgraph of $\Gamma$ generated by hyperplanes at distance exactly $n$ from $V$ in $\Gamma$.

\begin{defn}[Roots of a full sphere]\label{defn:root}
Let $\bar S_n(V)$ be a full sphere in $\Gamma$, with $n\geq 1$.  A \emph{grade-$n$ root} $C$ of $\bar S_n(V)$ is the full subgraph of $\Gamma$ generated by hyperplanes in $\bar S_n(V)^{(0)}\cap B$, where $B$ is a path-component of $\Gamma-\bar B_{n-1}(V)$.  The grade-0 root is the vertex corresponding to $V$.
\end{defn}

A root $C$ of $\bar S_n(V)$ is a union of path-components of $\bar S_n(V)$.  The 0-skeleta of the roots of $\bar S_n(V)$ may be regarded as equivalence classes, where hyperplanes $V$ and $W$ are equivalent if they are joined by a path in $\Gamma$ that contains no vertex of $\bar B_{n-1}(V)$.  The language of \emph{graded} hyperplanes defined below facilitates discussion of full spheres.

\begin{defn}\label{defn:grades}
Let $\Gamma$ be the contact graph of the CAT(0) cube complex $\widetilde X$.  With respect to a fixed base hyperplane $V^0$, the hyperplane $W$ has \emph{grade} $n$ if $W\in \bar S_n(V^0)$.  If $D\rightarrow\widetilde X$ is a disc diagram containing a dual curve $K$, the \emph{grade} of $K$ is the grade of the hyperplane to which $K$ maps.
\end{defn}

\subsection{Precursors, ancestors and footprints}
\emph{Precursors} are local features of $\Gamma$ that govern how concentric full spheres fit together, and footprints are related subspaces of $\widetilde X$ by which the presence of grade-$n$ hyperplanes are reflected in the grade-$(n-1)$ hyperplanes.  Ancestors are subcomplexes of $\widetilde X$ that contain precursors and footprints.  Precursors have an implicit role in the proof that $\Gamma$ is a quasi-tree.

\begin{defn}[Planar grid]
Let $\reals$ denote the real line, regarded as a cube complex with $\reals^{(0)}=\integers$.  An \emph{interval} $I$ is a nonempty connected subcomplex of $\reals$.  A \emph{planar grid} $S$ is a 2-dimensional CAT(0) cube complex isomorphic to $I\times I'$, where $I,I'$ are (possibly infinite) subdivided intervals.  Note that a planar grid is a convex subcomplex of $\reals\times\reals$.  Planar grids feature in a minor manner in Lemma~\ref{lem:precursors} and play an important role in Section~\ref{sec:planarbipartiteII}.
\end{defn}

\begin{defn}\label{defn:Precursors}
Fix a base hyperplane $V$ of $\widetilde X$ and grade the hyperplanes of $\widetilde X$ with respect to $V$. Let $U\in \bar S_n(V)$, with $n\geq 1$.  A \emph{precursor of $U$} is a hyperplane $W\in \bar S_{n-1}(V)$ such that $U\coll W$.  For $n\geq 1$, a \emph{common precursor} for an edge $U_1\coll U_2$ in $\bar S_n(V)$ is a vertex $W\in \bar S_{n-1}(V)$ such that any length-$n$ path from $V$ to $U_i$ passes through $W$, for $i=1,2$. For example, all edges of $S_1(V)$ have $V$ as a common precursor of their endpoints.

For $n\geq 2$, an \emph{edge-precursor} for an edge $U_1\coll U_2$ in $\bar S_n(V)$ is an edge $W_1\coll W_2$ in $\bar S_{n-1}(V)$ such that $U_i\coll W_i$ for $i=1,2$.
See Figure~\ref{fig:PrecDef}.
\begin{figure}[h]
  \includegraphics[width=1.75 in]{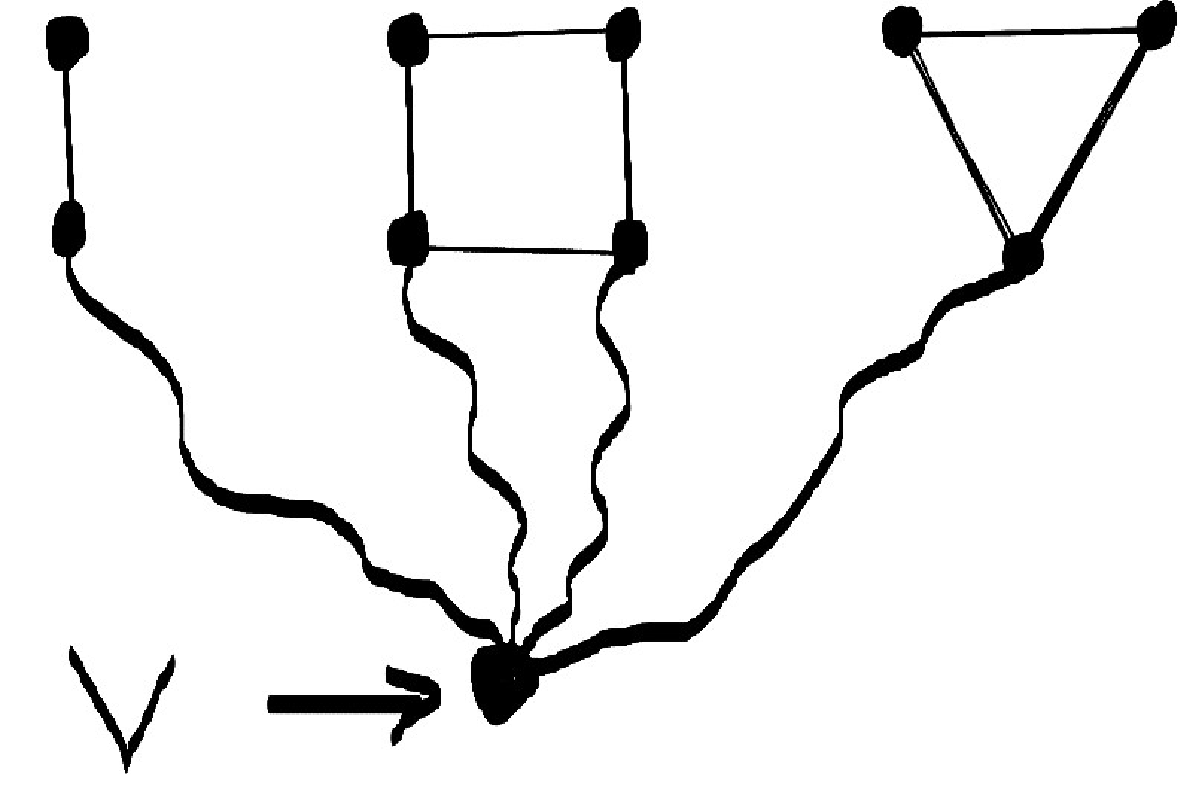}\\
  \caption{Left to right, in $\Gamma$: a precursor, an edge-precursor and a common precursor.  The wavy paths are geodesics of length $n-1$.}\label{fig:PrecDef}
\end{figure}
\end{defn}

The following lemma shows that edge-precursors and common precursors exist in $\Gamma$.  The edge in $S_{n-1}(V)$ defining an edge-precursor may arise as an osculation -- the analogous statement for crossing graphs is false.

\begin{lem}\label{lem:precursors}
Let $\widetilde X$ be a CAT(0) cube complex with contact graph $\Gamma$.  For $n\geq 2$, if $U^n_1$ and $U^n_2$ in $\bar S_n(V)\subset\Gamma$ are adjacent, then either they have a common precursor or the edge $U^n_1\coll U^n_2$ has an edge-precursor.
\end{lem}

\begin{proof}
Either the $U^n_i$ have a common precursor or there exist geodesic paths $\sigma_i$ in $\Gamma$, for $i=1,2$, which are concatenations $V=U^0_i\coll U^1_i\coll\ldots\coll U^{n-1}_i\coll U^n_i$ such that $U^j_i\in S_j(V)$ and $U^{n-1}_1\neq U^{n-1}_2$.  In the latter case, choose a closed path $\gamma\rightarrow\widetilde X$ that is a concatenation
\begin{equation*}
\gamma=P^0P^1_1P^2_1\ldots P^{n-1}_1P^n_1P^n_2P^{n-1}_2\ldots P^1_2,
\end{equation*}
where $P^j_i\rightarrow N(U^j_i)$ and $P^0\rightarrow N(V)$.  Let $D\rightarrow\widetilde X$ be a disc diagram with boundary path $\gamma$, and suppose that $D$ has minimal complexity among all such diagrams for all such choices of geodesic segments in $\Gamma$.  This situation is illustrated in Figure~\ref{fig:Precursors}.
\begin{figure}[h]
  \includegraphics[width=3 in]{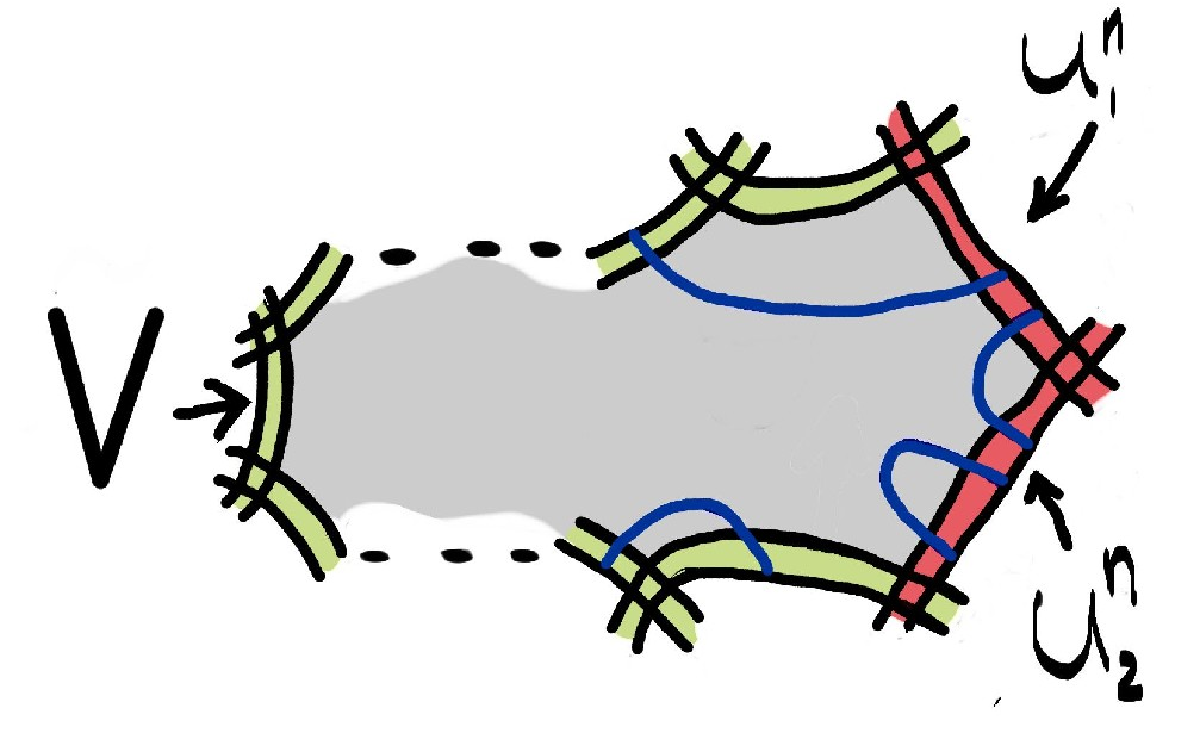}\\
  \caption{The disc diagram in Lemma~\ref{lem:precursors}.}\label{fig:Precursors}
\end{figure}
No dual curve in $D$ has both ends on a subpath of $\gamma$ that maps to a single hyperplane carrier, by minimality of area.

Let $C$ be a dual curve originating on $P^n_i$.  Since $U^n_i$ has grade $n$, the hyperplane $U$ to which $C$ maps cannot cross $U^j_k$ for $j<n-2$, so that $C$ must end on $P^j_k$ with $k=1,2$ and $j\geq n-2$.  If $C$ ends on $P^{n-1}_i$, then there is a lower-complexity choice of $D$ by Lemma~\ref{lem:fixedcorners}.  If $C$ ends on $P^{n-2}_i$, then the path $\sigma_i$ can be modified by replacing $U^{n-1}_i$ by $U$, leading to a lower-area disk diagram.  Hence $C$ ends on $P^j_k$ with $j=n-1$ or $n-2$ and $k\neq i$.

If $C$ ends on $P^{n-2}_k$, as on the left of Figure~\ref{fig:Precursors2}, then there are two possibilities.  If some dual curve $C'$ originating on $P^n_k$ ends on $P^{n-2}_i$, then the hyperplanes corresponding to $C$ and $C'$ are an edge-precursor for $U^n_1$ and $U^n_2$.
\begin{figure}[h]
  \includegraphics[width=4 in]{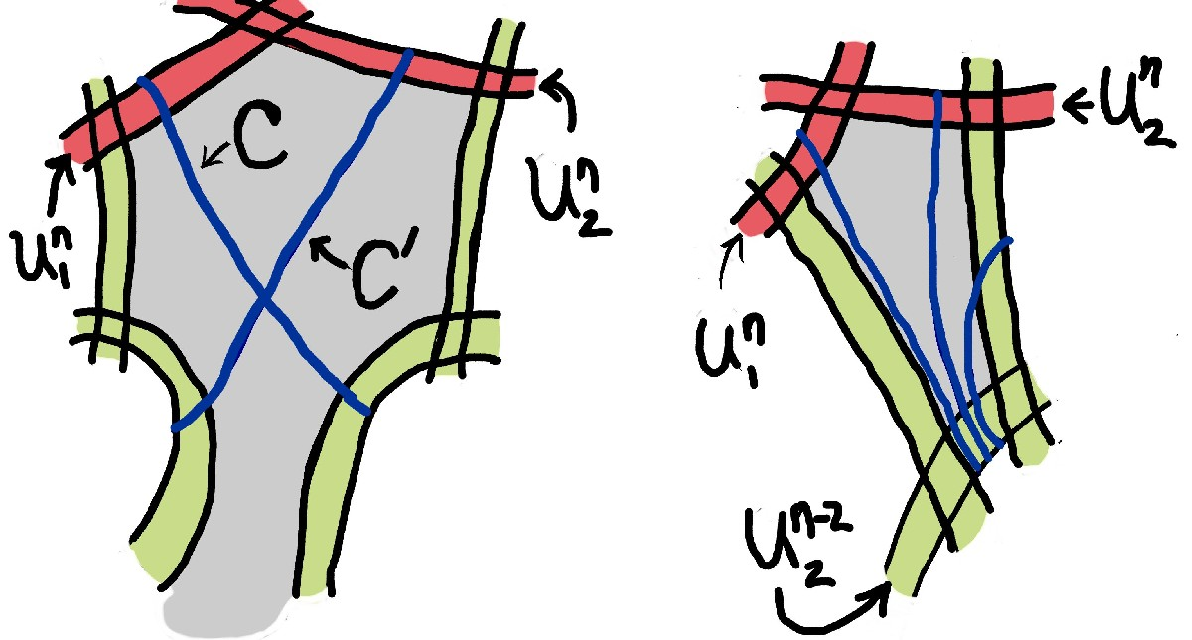}\\
  \caption{Obtaining a precursor-pair.  The dual curves at right are $C_1,C_2,C_3$.}\label{fig:Precursors2}
\end{figure}
If not, then observe that $\sigma_i$ can be replaced by the path $U^0_k\coll U^1_k\coll\ldots\coll U\coll U^n_i$, yielding a lower-area \emph{pentagonal diagram} $D'$ as on the right of Figure~\ref{fig:Precursors2}.  Any dual curve to the subpath of $P^{n-2}_k$ contained in $\partial_pD'$ leads to a contradiction: if such a dual curve $C_1$ ends on $P^n_i$, then area can be further decreased by using $C_1$ in place of $C$; if $C_2$ travels from $P^{n-2}_k$ to $P^n_k$ then replace $U^{n-1}_k$ by the hyperplane corresponding to $C_2$; if $C_3$ has any of the other two possible destinations, Lemma~\ref{lem:fixedcorners} gives a contradiction.  These possibilities are shown at right in Figure~\ref{fig:Precursors2}.  Hence the subtended part of $P^{n-2}_k$ is a trivial path, and $N(U)\cap N(U^{n-1}_k)\neq\emptyset$, so that those hyperplanes form an edge-precursor.
\begin{figure}[h]
    \includegraphics[width=3 in]{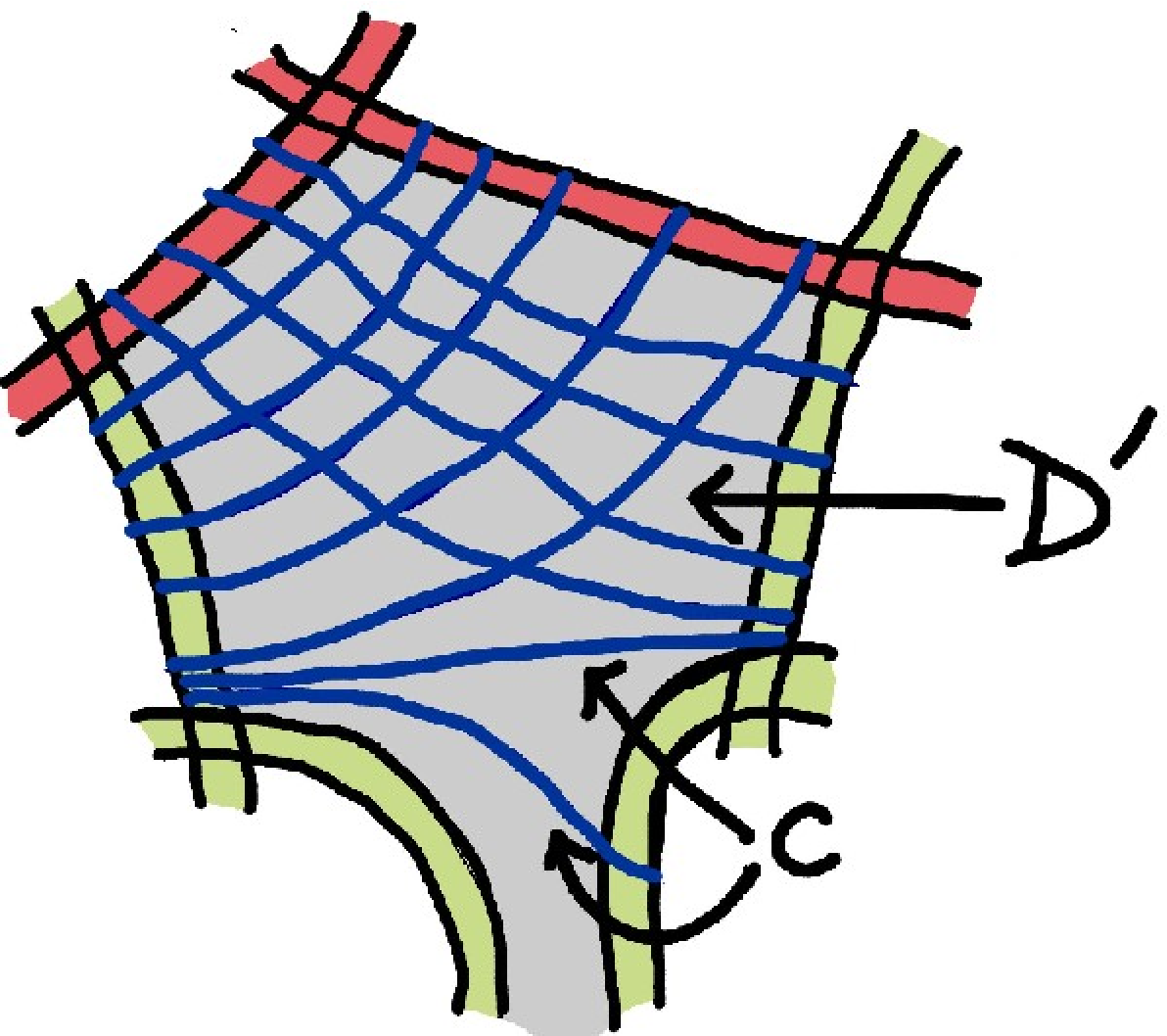}\\
    \caption{The grid case.}\label{fig:GridCase}
\end{figure}
The remaining possibility is that all dual curves emanating from $P^n_i$ end on $P^{n-1}_k$ and vice versa.  No two dual curves from $P^n_i$ or $P^{n-1}_i$ cross and thus there is thus a planar grid in $D$, as in Figure~\ref{fig:GridCase}.  An innermost dual curve $C$ to $P^{n-1}_i$ that does not end on $P^n_k$ forms part of the boundary path of a subdiagram $D'\subset D$, containing the planar grid, such that any dual curves in $D'$ emanating from $C$ have no possible destination.  Thus $U^{n-1}_1\coll U^{n-1}_2$.
\end{proof}

\begin{rem}
The analogue of Lemma~\ref{lem:precursors} does not hold for crossing graphs.  Consider a 10-gon tiled by squares consisting of 5 squares meeting around a 0-cube. Any choice of base hyperplane gives an adjacent pair of grade-2 (in the crossing graph) hyperplanes that do not have a common (crossing) precursor of grade 1 or an edge-precursor, since the grade-1 hyperplanes do not cross.
\end{rem}

Given a central hyperplane $V$ and a radius $n\geq 0$, there is a subcomplex $Y_n=\bigcup_{W^n} N(W^n)$ corresponding to $\bar S_n(V)$.  For $n\geq 1$, the subcomplex $Y_n\subset\widetilde X$ is not in general convex, but nonetheless exhibits some of the behavior of a convex subcomplex.

\begin{defn}[Ancestor]\label{defn:ancestor}
Given $U\in\bar S_n(V)^{(0)}$, the \emph{ancestor} $\ancestor(U)$ of $U$ is the subcomplex of $Y_{n-1}$ consisting of the union of all carriers $N(W)$ such that $W\in\bar S_{n-1}(V)^{(0)}$ and $U\coll V$.
\end{defn}

\begin{defn}[Footprint]\label{defn:footprint}
For $n\geq 1$, if $U\in \bar S_n(V)^{(0)}\subset\Gamma$, then the \emph{footprint $F(U)$ of $U$ in $\bar S_{n-1}(V)$} is the subspace
\begin{equation*}
    F(U)=\bigcup_{W\in S_{n-1}(V)^{(0)}}N(U)\cap N(W)
\end{equation*}
of $\ancestor(U)$.  Each intersection $N(U)\cap N(W)=F(U;W)$ is the \emph{footprint of $U$ in $W$}.
\end{defn}

The following lemmas enable statements about hyperplanes to be proven by induction on dimension, since they show that hyperplanes inherit the adjacency properties of their footprints.

\begin{lem}\label{lem:Connectedancestor}
For $U\in \bar S_n(V)^{(0)}$, the ancestor $\ancestor(U)$ and the footprint $F(U)$ are connected.
\end{lem}

\begin{proof}
If $n=1$, then the ancestor is the connected subcomplex $N(V)$ and the footprint $N(U)\cap N(V)$ is connected by convexity of hyperplane carriers.

Let $U^{n-1}_1$ and $U^{n-1}_2$ be distinct precursors of $U$.  For $i=1,2$, choose geodesics
\begin{equation*}
V=U^0_i\coll\ldots\coll U^{n-1}_i\coll U
\end{equation*}
in $\Gamma$.  As in Lemma~\ref{lem:precursors}, choose a closed path
\begin{equation*}
\gamma=P_0P^1_1P^2_1\ldots P^{n-1}_1QP^{n-1}_2\ldots P^1_2,
\end{equation*}
with $P_0\rightarrow N(V)$,  $P^j_i\rightarrow N(U^j_i)$ and $Q\rightarrow N(U)$.  Let $D\rightarrow\widetilde X$ be a disc diagram with boundary path $\gamma$, and suppose that the choice of precursors, of geodesics in $\Gamma$, of $\gamma$, and of $D$ are made so that $D$ has minimal complexity with respect to all these possibilities.  See Figure~\ref{fig:ConnAnc}.
\begin{figure}[h]
  \includegraphics[width=5 in]{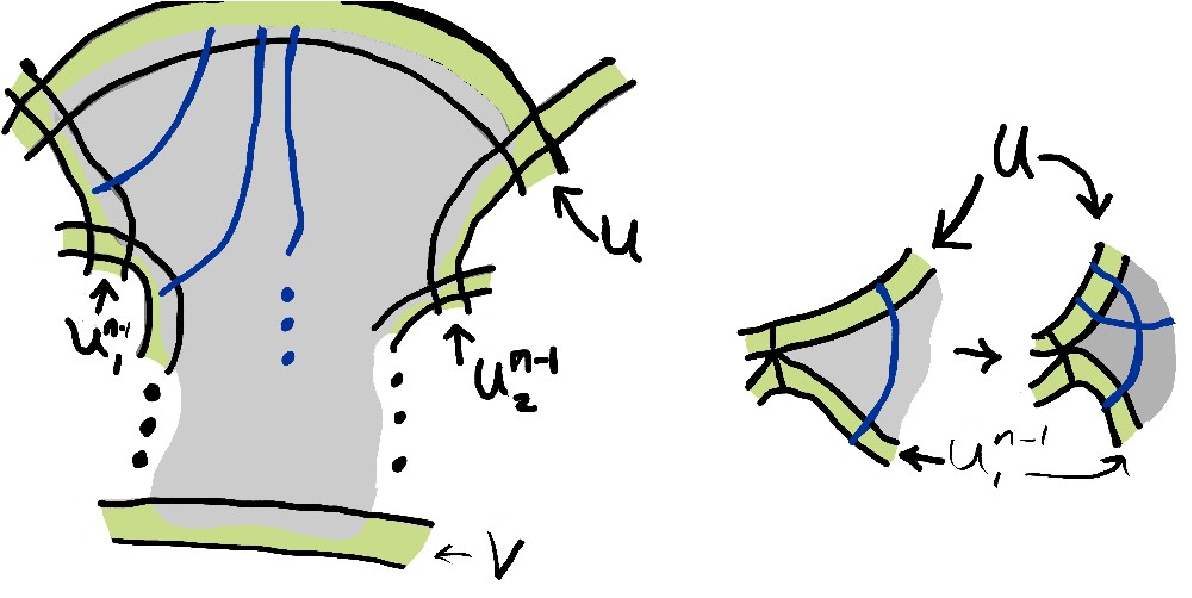}\\
  \caption{Ancestors are connected.}\label{fig:ConnAnc}
\end{figure}
Consider a dual curve $C$ in $D$ with an end on $Q$.  Every possibility for the other end of $C$ leads to a contradiction: two ends on $Q$ gives a bigon; an end on $P^{n-1}_i$ leads to a contradiction of Lemma~\ref{lem:fixedcorners} (the osculating case is shown at the right of Figure~\ref{fig:ConnAnc}); an end on $P^{n-2}_i$ leads to a modification of geodesic in $\Gamma$ resulting in an area reduction; an end on $P^{n-k}_i$ with $k>2$ contradicts the hypothesis that $U\in \bar S_n(V)$.  An end on $P_0$ leads to a closer pair of precursors and a choice of geodesic in $\Gamma$ that lowers area for $n=2$, and contradicts the fact that the chosen path in $\Gamma$ is geodesic if $n>2$. Hence $Q$ is a length-0 path, so that $N(U^{n-1}_1)\cap N(U^{n-1}_2)\neq\emptyset$.  The preceding argument also proves connectedness of $F(U)$.
\end{proof}

\begin{lem}\label{lem:FootprintAdjacency}
If $U_1,U_2\in \bar S_n(V)^{(0)}$ and $W\in S_{n-1}(V)$ is a common precursor, then $U_1\coll U_2$ if and only if $F(U_1;W)\cap F(U_2;W)\neq\emptyset$.
\end{lem}

\begin{proof}
This follows immediately from Lemma~\ref{lem:HellysTheorem}.
\end{proof}

\section{Contact graphs are quasi-trees}\label{sec:quasitree}
Fix a base hyperplane $V^0$ of $\widetilde X$.  For each $n\geq 0$, let $\mathcal C^n$ denote the set of grade-$n$ roots of the full sphere $\bar S_n(V^0)$.  Recall that a root $C\in\mathcal C^n$ is the full subgraph of $\Gamma$ generated by the vertices $V^n\in \bar S_n(V^0)$ with the property that any two $V_1^n,V_2^n\in C$ are joined by a path in $\Gamma-\bar B_{n-1}(V^0)$.  In particular, the graph $C$ may not be connected.

The main theorem in this section is Theorem~\ref{thm:HGisQT}, and we give two quite different proofs.

\begin{thm}\label{thm:HGisQT}
Let $\widetilde X$ be a CAT(0) cube complex with contact graph $\Gamma$.  Then $\Gamma$ is quasi-isometric to a tree.
\end{thm}

\begin{proof}[Proof of Theorem~\ref{thm:HGisQT} using the bottleneck criterion]
Manning's ``bottleneck'' criterion, introduced in \cite{ManningPseudocharacters}, is as follows:\\

\textit{The geodesic metric space $(Y,d)$ is quasi-isometric to a simplicial tree if and only if there exists $\delta>0$ such that, for any two points $x,y\in Y$, there exists a midpoint $M=M(x,y)$ such that $d(M,x)=d(M,y)=\frac{1}{2}d(x,y)$ and any path joining $x$ to $y$ contains a point within $\delta$ of $M$.}\\

Let $V_0,V_n$ be hyperplanes, and let $\{V_i\}_{i=1}^{n-1}$ be the set of hyperplanes $V_i$ such that $V_0$ and $V_n$ lie in distinct halfspaces associated to $V_i$, i.e. the \emph{set of hyperplanes separating $V_0$ and $V_n$}.  Then, for each $i$, any path in $\Gamma$ joining $V_0$ to $V_n$ must either contain $V_i$ or contain some hyperplane that crosses $V_i$.  Indeed, if $V_0=W_0\coll W_1\coll\ldots\coll W_m=V_n$ is a path, then there exists a path $Q=Q_0Q_1\ldots Q_n$ in $\widetilde X$, where each $Q_j$ lies in $N(W_j)$.  Now, since $V_i$ separates $W_0$ from $W_m$, the path $Q$ must contain a 1-cube $c$ dual to $V_i$.  For some $j$, the path $Q_j$ contains $c$.  Either $W_j=V_i$ and $c$ is dual to $W_j$, or $c\subset N(W_j)$ is not dual to $W_j$, and hence $V_i$ and $W_j$ cross.

Let $V_0=U_0\coll U_1\coll\ldots\coll U_m=V_n$ be a geodesic path in $\Gamma$ joining $V_0$ to $V_n$.  Let $M$ be the midpoint of this path, so that either $M=U_{m/2}$ or $M$ is the midpoint of $U_{(m-1)/2}\coll U_{(m+1)/2}$, according to the parity of $m$.  Now, there exists $i\leq n-1$ such that $d_{\Gamma}(M,V_i)\leq\frac{3}{2}$.  Indeed, either $U_{m/2}=M$ is equal to, or crosses, some $V_i$, or $U_{(m\pm1)/2}$ is equal to, or crosses, some $V_i$, by Lemma~\ref{lem:crossseparators}.  Here we have assumed that $m\geq 2$, in order to apply Lemma~\ref{lem:crossseparators}.  If $m=1$, then $U_0\coll U_m$, and the midpoint $M$ of this edge obviously satisfies Manning's criterion with $\delta=\frac{1}{2}$.

Let $V_0=W_0\coll W_1\coll\ldots\coll W_p=V_n$ be some other path joining $V_0$ to $V_n$.  Then some $W_j$ either crosses or coincides with $V_i$, so that $d_{\Gamma}(W_j,M)\leq d_{\Gamma}(W_j,V_i)+d_{\Gamma}(V_i,M)\leq\frac{5}{2}$.  Thus Manning's criterion is verified with $M(V_0,V_n)=M$ and $\delta=\frac{5}{2}$.
\end{proof}

\begin{lem}\label{lem:crossseparators}
Let $U_0\coll U_1\coll\ldots\coll U_m$ be a geodesic of $\Gamma$, and let $\{V_i\}_{i=0}^{n-1}$ be the set of hyperplanes separating $U_0$ from $U_m$.  Then for each $j$ with $1\leq j\leq m-1$, there exists $i$ such that $d_{\Gamma}(U_j,V_i)\leq 1$.
\end{lem}

\begin{proof}
Let $Q$ be a geodesic segment beginning on $N(U_0)$ and ending on $N(U_m)$, chosen as short as possible, so that the set of hyperplanes that are dual to 1-cubes of $Q$ is exactly $\{V_i\}$.  For $0\leq k\leq m$, let $P_k\rightarrow N(U_k)$ be a geodesic segment, chosen so that the $P_k$ are concatenable, i.e. there is a path $P=P_0P_1\ldots P_m\rightarrow\widetilde X$, and suppose that $P$ joins the endpoints of $Q$.  Let $D\rightarrow\widetilde X$ be a disc diagram bounded by $P$ and $Q$, and suppose that the choices of $P,Q,D$ are made so that $\area(D)$ is as small as possible.  Choose $j\leq m$, and consider the dual curves in $D$ that emanate from $P_j$.  There must be at least one such dual curve, since otherwise $U_{j-1}\coll U_{j+1}$, contradicting the fact that $U_0\coll\ldots\coll U_m$ is a geodesic.  If $K$ is a dual curve that emanates from $P_j$ and ends on $Q$, then $K$ maps to a hyperplane $V_i$ that separates $U_0$ from $U_p$, whence $d_{\Gamma}(U_j,V_i)=1$.  In this case, the proof is complete: we have found a separating hyperplane $V_i$ that crosses $U_j$.

Otherwise, each such $K$ ends on $P_s$ for some $s\leq m$.  Now, $s\neq j$, since $P_j$ is a geodesic, and $|s-j|\neq1$, since otherwise, if $K$ traveled from $P_j$ to $P_{j+1}$, we could perform hexagon moves and removal of spurs to choose a smaller choice of $D$, fixing the carriers $\{N(U_k)\}$.  On the other hand, if $K$ were to travel from $P_j$ to $P_{s}$ with $|s-j|\geq 3$, then there would be a path $U_s\coll W\coll U_j$ in $\Gamma$, where $W$ is the hyperplane to which $K$ maps, showing that $d_{\Gamma}(U_j,,U_s)\leq 2$, a contradiction.  Hence assume that every dual curve emanating from $P_j$ ends on $P_{j+2}$ or $P_{j-2}$.  Moreover, by minimality of area, no two such dual curves cross.  Now, label the dual curves $K_1,\ldots,K_c$ so that $K_q$ is dual to the $q^{th}$ 1-cube of $P_j$, measuring from $P_j\cap P_{j-1}$ to $P_{j}\cap P_{j+1}$.  If each $K_q$ ends on $P_{j-2}$, then $U_{j-2}\coll W_c\coll U_{j+1}$, where $W_c$ is the hyperplane to which $K_c$ maps, and this is a contradiction.  Similarly, if each $K_q$ ends on $P_{j+2}$, then $U_{j-1}\coll W_1\coll U_{j+2}$, another contradiction.  Hence we have $1<a<c$ such that $K_1,\ldots,K_a$ travel from $P_j$ to $P_{j-2}$ and $K_{a+1},\ldots,K_c$ travel from $P_j$ to $P_{j+2}$.  But then $U_{j-2}\coll W_a\coll W_{a+1}\coll U_{j+2}$, contradicting the fact that $d_{\Gamma}(U_{j-2},U_{j+2})=4$.  Hence some $K_q$ must end on $Q$ and map to a hyperplane separating $U_0$ from $U_m$.  See Figure~\ref{fig:bottleneck}.
\end{proof}

\begin{figure}[h]
  \includegraphics[width=3in]{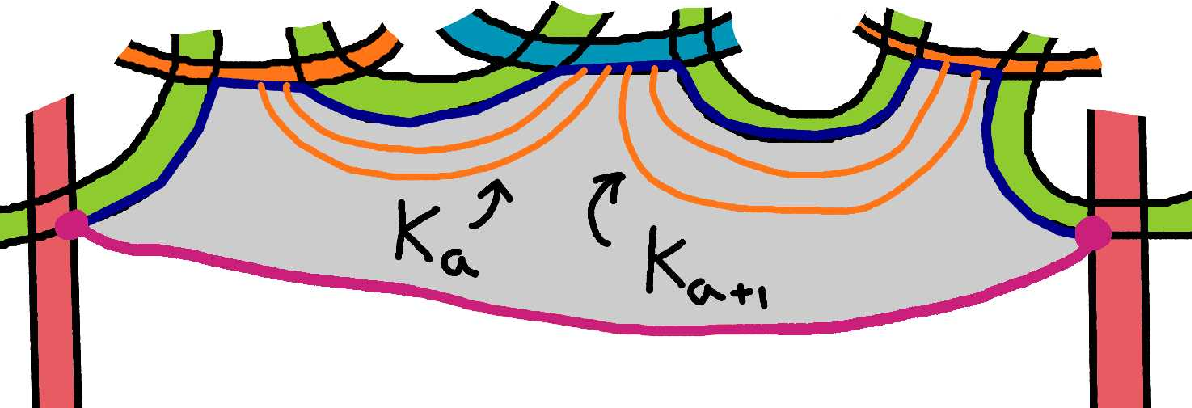}\\
  \caption{The dual curves emanating from $P_j$ cannot all end on $P$.}\label{fig:bottleneck}
\end{figure}

Theorem~\ref{thm:HGisQT} is also provable using disc diagrams:

\begin{proof}[Proof of Theorem~\ref{thm:HGisQT} using hyperplane grading]
Fix a base vertex $V^0$ of $\Gamma$.  The resulting \emph{graded root-tree} $\mathcal T$ is the following graph.  The 0-skeleton of $\mathcal T$ is the set $\coprod_{n\geq 0}\mathcal C^n$.  Edges join vertices in $\mathcal C^n$ to vertices in $\mathcal C^{n+1}$.  Precisely, if $C^n\in\mathcal C^n$ and $C^{n+1}\in\mathcal C^{n+1}$, then $C^n$ is adjacent to $C^{n+1}$ if and only if $C^n$ contains a vertex of $\Gamma$ that is adjacent to a vertex of $C^{n+1}\subset\Gamma$.

$\mathcal T$ is a tree.  To see this, note that for each $n\in\naturals$, no two vertices in $\mathcal C^n$ are adjacent, so that the presence of a cycle in $\mathcal T$ implies that for some $n$, there is a $C^{n+1}\in\mathcal C^{n+1}$ that is adjacent to two distinct vertices $C_1^n,C_2^n\in\mathcal C^n$.  It follows that there are hyperplanes $V_i^{n+1}\in C^{n+1},\,V_i^n\in C_i^n$ for $i=1,2$ such that $V_i^n\coll V_i^{n+1}$.  By definition, $V_1^{n+1}$ and $V_2^{n+1}$ are joined by a path in $\Gamma-\bar B_n(V^0)$, and we thus have a path in $\Gamma-\bar B_{n-1}(V^0)$ joining the $V_i^n$, so that $C_1^n=C_2^n$, a contradiction.

$\Gamma$ is quasi-isometric to $\mathcal T$.  Indeed, consider the map $\phi:\Gamma\rightarrow\mathcal T$ such that $\phi$ sends each hyperplane $V^n$ to the unique root of $\bar S_n(V^0)$ containing it, and does likewise for edges that have both endpoints in the same full sphere.  The remaining edges of $\Gamma$ join hyperplanes in roots of $\bar S_n(V^0)$ to $\mathcal T$-adjacent roots of $S_{n+1}(V^0)$, for $n\geq 0$.  These edges map isometrically to the corresponding edges of $\mathcal T$.  The map $\phi$ is surjective and a quasi-isometric embedding by Lemma~\ref{lem:HyperplaneGraphSphericalComp}.
\end{proof}

Lemma~\ref{lem:HyperplaneGraphSphericalComp} asserts the existence of a uniform bound on the diameters of the roots in the graded graph $\Gamma$.  

\begin{lem}\label{lem:HyperplaneGraphSphericalComp}
There exists a constant $M$ such that for any $n\geq 0$ and any base hyperplane $V^0$, if $C\in\mathcal C^n$, then $\diam_{\Gamma}(C)\leq M$.
\end{lem}

\begin{proof}
Argue by induction on the grade $n$ of $C=C^n$.  Since $\mathcal T$ is a tree, there is a unique sequence $C^0,C^1,\ldots,C^n$ of roots joining $C^0$ to $C^n$, i.e. for $0\leq i\leq n-1$, if $V^{i+1}\in C^i$ and $V^i$ is a precursor of $C^{i+1}$, then $V^i\in C^i$.

Let $V_1^n,V_2^n\in C^n$.  By definition, there is a path
\[\rho=V_1^n\coll U_1\coll U_2\ldots\coll U_m\coll V_2^n\]
in $\Gamma$ of minimal length so that $U_i$ has grade at least $n$ for $1\leq i\leq m$.  For $i\in\{1,2\}$, choose $\Gamma$-geodesics
\[\sigma_i=V^0\coll V_i^1\coll\ldots\coll V_i^{n-1}\coll V_i^n\]
joining $V^0$ to $V_i^n$.  Note that for each $k\leq n$, the hyperplane $V_i^k$ has grade exactly $k$ and lies in $C^k$.  For each $i\in\{1,2\}$ and each $k\leq n$, choose a geodesic segment $P_i^k\rightarrow N(V_i^k)$, and for $1\leq j\leq m,$ choose a geodesic segment $Q_j\rightarrow N(U_j)$ so that the above geodesics are concatenable, i.e. there is a closed path
\[P=P_0P_1^1P_1^2\ldots P^n_1Q_1Q_2\ldots Q_mP^n_2\ldots P_2^1\]
mapping to $\widetilde X$ and bounding a disc diagram $D\rightarrow\widetilde X$ with fixed carriers for the given hyperplanes. Suppose that $P$ and $D$ are chosen so that $D$ has minimal complexity for all such diagrams with those fixed carriers.  Moreover, suppose that paths $\rho$ and $\sigma_i$ joining $V^0,V_1^n,V_2^n$ are chosen so that $D$ has minimal area among all such minimal complexity fixed-carrier diagrams constructed in this way.  See Figure~\ref{fig:Nis32}.

\begin{figure}[h]
  \includegraphics[width=2in]{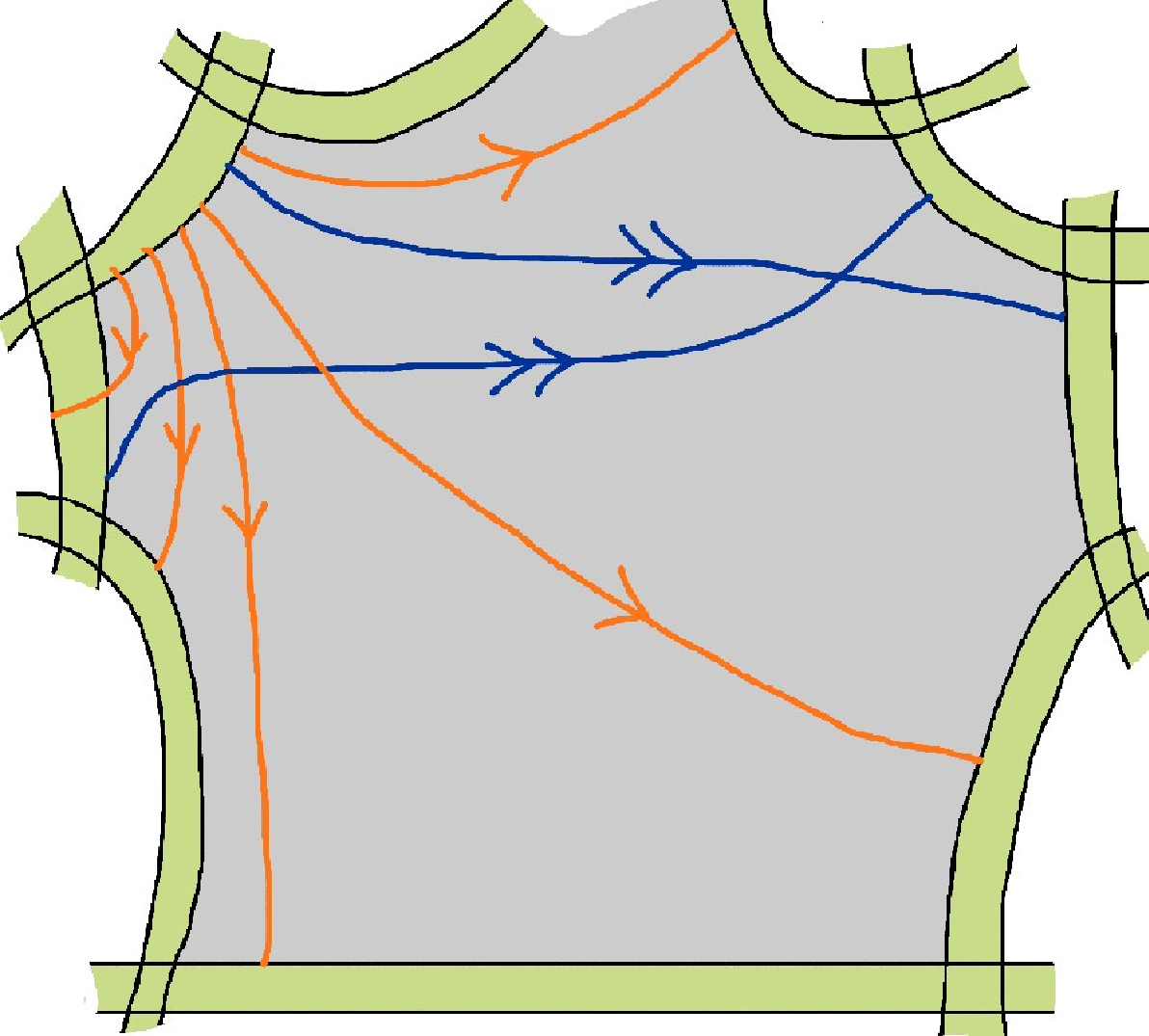}\\
  \caption{The case $n=3$.  Double-arrowed dual curves are possible and single-arrowed ones lead to various contradictions.}\label{fig:Nis32}
\end{figure}

Observe that each path $P_i^k$ for $1\leq k\leq n-1$ has length at least 1, for otherwise $V_i^{k-1}\coll V_i^{k+1}$, contradicting the grading.  Similarly, each $Q_j$ has length at least 1, for otherwise $U_{j-1}\coll U_{j+1}$ (or, e.g. $V_1^n\coll U_2$ when $j=1$), contradicting the minimum-length assumption on $\rho$.

Suppose that $n\geq 3$.  Then $|P_0|=0$.  Indeed, if $K$ is a dual curve emanating from $P_0$, then $K$ cannot end on any $Q_j$, since $K$ maps to a grade-1 hyperplane.  Similarly, $K$ can't end on $P_i^k$ for $k>2$.  If $K$ ends on $P_i^2$, then we can replace $D$ with a proper subdiagram by replacing $V_i^1$ by the hyperplane to which $K$ maps.  If $K$ ends on $P_i^1$, then we apply Lemma~\ref{lem:fixedcorners} to produce a lower-complexity fixed-carrier diagram.  If $K$ has two ends on $P_0$, then $P_0$ is not a geodesic, a contradiction.  Hence $K$ cannot exist, so $|P_0|=0$ and $V_1^1\coll V_2^1$.

When $n=3$, we thus have a path $V_1^3\coll V_1^2\coll V_1^1\coll V_2^1\coll V_2^2\coll V_2^3$ of length 5 joining $V_1^3$ to $V_2^3$.  For $n\leq 3$, it is thus evident that $d_{\Gamma}(V_1^j,V_2^k)\leq 4$ when $1\leq j<k\leq n$.

Now, with respect to $V_1^1$, the hyperplanes $V_1^n$ and $V_2^{n-1}$ have grade $n-1$ and lie in the same root based at $V_1^1$.  By induction, there is a path $V_1^n\coll W_1\coll\ldots\coll W_d\coll V_2^{n-1}$ with $d\leq 4$.  Hence for some $W_p$, the hyperplanes $V_1^n$ and $V_2^n$ lie in the full closed 3-ball about $W_p$ and thus $d_{\Gamma}(V_1^n,V_2^n)\leq 5$, by applying the same argument to the full 3-ball centered at $W_p$.
\end{proof}



In fact, the property of $\Gamma$ stated in Lemma~\ref{lem:HyperplaneGraphSphericalComp} characterizes graphs quasi-isometric to trees \cite{KroenMoeller08}.  This fact was proved independently in~\cite{CDNRV}.

\section{Weak hyperbolicity of cubulated groups and quasi-arboreal groups}\label{sec:weakrel}
It follows from Theorem~\ref{thm:HGisQT} that cocompactly cubulated groups satisfy a strong form of weak hyperbolicity, in which the coned-off Cayley graph with respect to hyperplane stabilizers is not only $\delta$-hyperbolic but is actually a quasi-tree.

\subsection{Weak hyperbolicity and quasi-arboreality.}
Farb defined a notion of relative hyperbolicity in terms of a \emph{coned-off Cayley graph} in which the peripheral subgroups are associated to cone-points.  While the additional property of \emph{bounded coset penetration} is needed to actually obtain relative hyperbolicity in the sense of Gromov \cite{Gromov87}, the following notion of \emph{weak hyperbolicity} is of interest.

\begin{defn}[Weak hyperbolicity,~\cite{FarbThesis}]\label{defn:weakrel1}
Let $G$ be a finitely generated group and $\left\{G_W\right\}$ a finite collection of subgroups.  Let $\Gamma$ be the graph obtained from the Cayley graph of $G$ with respect to some finite generating set as follows.  To the Cayley graph, add a vertex $gG_W$ for each distinct coset of each $G_W$, and join each $gG_W$ by an edge to each vertex of the Cayley graph corresponding to an element of $gG_W$.  The graph $\Gamma$ is the \emph{coned-off Cayley graph} of $G$ relative to $\{G_W\}$.  If there exists $\delta$ such that $\Gamma$ is $\delta$-hyperbolic, then $G$ is \emph{weakly hyperbolic relative to the collection} $\{G_W\}$.
\end{defn}

Bowditch gave another definition, in which the coned-off Cayley graph is replaced by a $G$-graph with similar properties.

\begin{defn}[Weak hyperbolicity \cite{Bowditch97}]\label{defn:relhyp}
Let $G$ be a group and $\{G_W\}$ a finite collection of subgroups.  $G$ is \emph{weakly hyperbolic} relative to $\{G_W\}$ if $G$ acts by isometries on a graph $\Gamma$ with the following properties:
\begin{enumerate}
  \item $\Gamma$ is $\delta$-hyperbolic for some $\delta$.
  \item There are finitely many $G$-orbits of edges.
  \item Each $G_W$ fixes a vertex of $\Gamma$ and each vertex stabilizer contains a conjugate of some $G_W$ as a subgroup of finite index.
\end{enumerate}
A $G$-graph $\Gamma$ satisfying the latter two properties is a \emph{generalized coset graph} for the pair $(G,\{G_W\})$, so that weak hyperbolicity amounts to the existence of a $\delta$-hyperbolic generalized coset graph.
\end{defn}

A stronger property is:

\begin{defn}[Quasi-arboreal group]
Let $G$ be a group and $\{G_W\}$ a finite collection of subgroups for which there is a generalized coset graph $\Gamma$ such that $\Gamma$ is quasi-isometric to a tree.  Then $G$ is \emph{quasi-arboreal relative to the collection} $\{G_W\}$.
\end{defn}

\subsection{Quasi-arboreality and cones on hyperplanes.}
Let $G$ be a finitely generated group acting on the CAT(0) cube complex $\widetilde X$.  Then $G$ acts on the contact graph $\Gamma$ by isometries, and the stabilizer of each vertex of $\Gamma$ is exactly the stabilizer of the corresponding hyperplane.

The following discussion is therefore extraneous to the proof of Corollary~\ref{cor:weakrelredux}, but gives a concrete viewpoint on the contact graph.
Let $\widetilde X$ be a CAT(0) cube complex with a set $\mathcal W$ of hyperplanes.  The \emph{coned-off complex} $\widetilde X^*$ is obtained from $\widetilde X$ by adding a cone on $N(W)$ for each $W\in\mathcal W$.  More precisely,

\begin{equation*}
    \widetilde X^*=\widetilde X\sqcup\left(\coprod_{W\in\mathcal W}N(W)\times[-1,1]\right)\diagup\left\{N(W)\times\{1\},\,N(W)\sim N(W)\times\{-1\}\right\}.
\end{equation*}

Associated to each hyperplane is a \emph{cone-point}, which is joined by a \emph{cone-edge} to each 0-cube in the corresponding hyperplane carrier.

The \emph{coned-off hyperplane graph} is $C(\widetilde X)=\left(\widetilde X^*\right)^{(1)}$.  When endowed with the combinatorial metric, $C(\widetilde X)$ is quasi-isometric to $\Gamma$, and to $\widetilde X^*$ when $\widetilde X$ is finite-dimensional.

Indeed, choose a map $\Gamma\rightarrow C(\widetilde X)$ that sends each vertex to the cone-point over the corresponding hyperplane.  Each edge joins a pair of vertices corresponding to a pair of cone-points joined by a path in $C(\widetilde X)$ that is a concatenation of two cone-edges.  Each edge of $\Gamma$ maps linearly to some such length-2 path, giving a $(2,0)$ quasi-isometric embedding $\Gamma\rightarrow C(\widetilde X)$.  To see this, let $U,V$ be hyperplanes, and let $u,v$ be the corresponding cone-points of $C(\widetilde X)$.  For any path $U=W_1\coll\ldots\coll W_n=V$ in $\Gamma$ joining $U$ to $V$, there exists a path $P_1P_2\ldots P_{n-1}$ in $C(\widetilde X)$ joining $u$ to $v$ and having length $2(n-1)$.  Indeed, $P_i$ is a path consisting of two consecutive cone-edges, traveling from the cone-point associated to $W_i$ to that associated to $W_{i+1}$ via a 0-cube of $N(W_i)\cap N(W_{i+1})$, which exists by the definition of contacting hyperplanes.  Hence $d_{C(\widetilde X)}(u,v)\leq 2d_{\Gamma}(U,V)$.

On the other hand, let $P$ be a geodesic path joining $u$ to $v$, so that
\[P=K'_1B_1K_2B_2\ldots K_nB_nK'_{n+1},\]
where each $B_i$ is a (possibly trivial) path in $\widetilde X^{(1)}$, each $K_i$ is a concatenation of two cone-edges containing a single cone-point, and $K'_1,K'_{n+1}$ are single cone-edges.  Now, let $c$ be a length-1 subpath of some $B_i$.  Then there exists a path $EF$, consisting of cone-edges, with the same endpoints as $c$.  Indeed, $EF$ travels from the initial point of $c$, through the cone-point associated to the hyperplane dual to $c$, and ends at the terminus of $c$.  Hence there is a path of length $m=\sum_i|K_i|+2\sum_i|B_i|+2$ in $C(\widetilde X)$ that joins $u$ and $v$ and consists entirely of cone-edges; this path is obtained by replacing each 1-cube of each $B_i$ by a length-2 cone-path in the preceding manner.  This path has the form $E_1F_1\ldots E_kF_k$, where each $E_i$ travels from a cone-point to a 0-cube of $\widetilde X^{(1)}$, and each $F_i$ travels from a 0-cube to a cone-point.  Let $U_i$ be the hyperplane corresponding to the cone-point that is the initial point of $E_i$ (and terminal point of $F_{i-1}$).  Then $U_1\coll U_2\coll\ldots\coll U_{k+1}$ is a path in $\Gamma$ joining $U$ to $V$ and having length $k=m/2$.  Now, $|P|=\sum_i|K_i|+\sum_i|B_i|+2\geq m/2$, so that $d_{C(\widetilde X)}(u,v)\geq d_{\Gamma}(U,V)$.

Since every point of $\widetilde X$ lies in some hyperplane carrier, every point of $C(\widetilde X)$ lies at distance at most $\frac{3}{2}$ from some cone-point, so that the map is quasi-surjective.  Thus $C(\widetilde X)$ is a quasi-tree by Theorem~\ref{thm:HGisQT}.  That $\widetilde X^*$ is quasi-isometric to $C(\widetilde X)$ when $\widetilde X$ is finite-dimensional follows easily from the fact that, if $\dimension\widetilde X<\infty$, then $\widetilde X$ and $\widetilde X^{(1)}$ are quasi-isometric (see e.g.~\cite{CapraceSageev} for a proof).

Let $G$ be a group acting properly and cocompactly on the CAT(0) cube complex $\widetilde X$.  Then $G$ acts by isometries on $C(\widetilde X)$, with this action extending that of $G$ on $\widetilde X^{(1)}$.  The stabilizer of a vertex of $C(\widetilde X)$ is finite when the vertex is a 0-cube of $\widetilde X$ and equal to $G_W$, the stabilizer of the hyperplane $W$, for the vertex corresponding to $W$.  The 1-cubes of $\widetilde X$ have finite stabilizers by properness, and the cone-edges are finitely stabilized since they each have an initial vertex that is a 0-cube.  Moreover, by cocompactness, there are finitely many orbits of edges.

\begin{cor}\label{cor:weakrelredux}
Let $G$ act on the CAT(0) cube complex $\widetilde X$.  Then $G$ acts on a graph $\Gamma$ that is quasi-isometric to a tree, such that the stabilizers of hyperplanes in $\widetilde X$ correspond to the stabilizers of vertices in $\Gamma$.

Furthermore, let $G\cong\pi_1X$, with $X$ a nonpositively-curved cube complex with $\mathcal W$ the set of immersed hyperplanes in $X$.  Suppose that $\mathcal W$ is finite and that there are finitely many contacts between immersed hyperplanes in $X$.  (For instance, these hypotheses are satisfied when $X$ is compact.)  Then $G$ is quasi-arboreal relative to the set $\left\{\pi_1W\,:\,W\in\mathcal W\right\}$.
\end{cor}

\begin{exmp}
The following groups act on quasi-trees by virtue of their actions on CAT(0) cube complexes.
\begin{enumerate}
  \item Finitely presented groups satisfying the $B(4)-T(4)$ small-cancellation condition act properly and cocompactly on CAT(0) cube complexes, and $B(6)$ groups act properly on CAT(0) cube complexes with finitely many orbits of hyperplanes \cite{WiseSmallCanCube04}.
  \item A right-angled Artin group $R$ acts properly discontinuously and cocompactly on a CAT(0) cube complex that consists of Euclidean spaces of various dimensions, tiled by cubes, attached along affine subspaces \cite{CharneyDavis94}.  The hyperplane stabilizers are themselves right-angled Artin groups.
  \item Farley proved that Thompson's group $V$ acts properly discontinuously on a CAT(0) cube complex with two orbits of hyperplanes, one of which consists of trivially stabilized hyperplanes.  Hence $V$ acts on a quasi-tree $\Gamma$.  More generally, Farley gave an action on a CAT(0) cube complex for \emph{diagram groups} associated to based semigroup presentations \cite{Farley2003},\cite{Farley2005}.
  \item Finitely generated Coxeter groups act properly on CAT(0) cube complexes with finitely many orbits of hyperplanes \cite{NibloReeves03}.
  \item Artin groups of type FC act on finite dimensional CAT(0) cube complexes with 0-cube stabilizers of finite type~\cite{CharneyDavis95b}.
\end{enumerate}
\end{exmp}

The next example shows that there are non-cubulated quasi-arboreal groups.

\begin{exmp}\label{exmp:notcubulated}
Let $G\cong N\rtimes F$ where $F$ is a finitely-generated free group and $N$ is a finitely-generated group.  Let $\Gamma$ be the graph whose vertices are distinct cosets of $N$ and whose edges correspond to left-multiplication by generators of $G/N\cong F$.  Then $G$ acts on $\Gamma$ in such a way that the vertex-stabilizers are all $N$ and the set of $G$-orbits of edges generates $F$.  In fact, $\Gamma$ is a Cayley graph for $F$ and is thus a tree.  $\Gamma$ is also a generalized coset graph showing that $G$ is quasi-arboreal relative to $N$.

$N$ and $F$ may be chosen in such a way that $G$ does not act properly on a CAT(0) cube complex.  For instance, let $G$ be the Baumslag-Solitar group with presentation $\langle a,b\mid (a^m)^b=a^n\rangle$.  Then $G$ is weakly hyperbolic relative to $\langle a\rangle$ with generalized coset graph a subdivided line.  However, a theorem of Haglund in \cite{HaglundSemisimple} implies that $G$ is not cubulated when $m\neq n$.
\end{exmp}

\begin{defn}
Let $G$ be a finitely generated group and $\mathcal G$ its Cayley graph with respect to some finite generating set.  A subgroup $H\leq G$ is a \emph{codimension-1 subgroup} if there exists $r\geq 0$ such that $\mathcal G- N_r(H)$ has two components, neither of which lies in $N_s(H)$ for any $s\geq 0$.
\end{defn}

One verifies that, given an action of $G$ on a CAT(0) cube complex, the hyperplane-stabilizers are codimension-1 subgroups.  Conversely, Sageev's construction yields an action of $G$ on a CAT(0) cube complex in the presence of a codimension-1 subgroup.  A ready class of examples of groups without codimension-1 subgroups is that of groups having Kazhdan's Property~(T)~\cite{NibloRoller98}, and the following example shows that quasi-arboreality does not imply the existence of a codimension-1 subgroup.

\begin{exmp}\label{exmp:slnz}
Consider the Steinberg presentation for $SL_n(\integers)$, with $n\geq 3$, where the generator $a_{ij}$ represents the $n\times n$ matrix with diagonal entries equal to 1, the $ij$-entry equal to 1, and 0 elsewhere:
\begin{equation*}
    SL_n(\integers)\cong\big\langle a_{ij},\,1\leq i\neq j\leq n\,\mid\,[a_{ij},a_{kl}],i\neq k, j\neq l;\,[a_{ij},a_{jk}]a_{ik}^{-1},i\neq k;\,(a_{12}a_{21}a_{12}^{-1})^4\big\rangle.
\end{equation*}
Let $A_{ij}=\langle a_{ij}\rangle$ and denote by $\Gamma$ the coned-off Cayley graph of the pair $\left(SL_n(\integers),\{A_{ij}\}\right)$.  A theorem of Carter and Keller implies that $SL_n(\integers)$ is boundedly generated with respect to $\{A_{ij}\}$~\cite{CarterKeller}.  The graph $\Gamma$ is therefore bounded, and hence $SL_n(\integers)$ is quasi-arboreal relative to $\{A_{ij}\}$.  On the other hand, $SL_n(\integers)$ has Property~(T)~\cite{HarpeValette89} and thus contains no codimension-1 subgroups.
\end{exmp}

\section{Asymptotic dimension}\label{sec:asdim}

\subsection{Asymptotic dimension of cube complexes}
In this section, we discuss the asymptotic dimension of groups acting on CAT(0) cube complexes and relate this to quasi-arboreality.

\begin{defn}[Asymptotic dimension, \cite{BellDranishnikov}]\label{defn:asymdim}
Let $(M, d)$ be a metric space.  The \emph{asymptotic dimension} of $M$ is at most $n$ if for each $r>0$ there exists a covering $M=\cup_{i\in I}U_i$ such that the sets $U_i$ are uniformly bounded and no more than $n+1$ elements of $\{U_i\}_{i\in I}$ intersect any ball of radius $r$.

If $\asdim M\leq n$ and $\asdim M\not\leq n-1$, then we say $\asdim M=n$.  If no such $n$ exists, then $M$ is \emph{asymptotically infinite-dimensional}.
\end{defn}

The asymptotic dimension of a metric space is a quasi-isometry invariant and is thus well-defined for finitely-generated groups.  Word-hyperbolic groups have finite asymptotic dimension \cite{Gromov93}, but whether this is true of all CAT(0) groups is unknown.

Other examples of groups with finite asymptotic dimension are those that split as finite graphs of groups whose vertex-groups have finite asymptotic dimension \cite{BellDranishnikov01} and groups that are hyperbolic relative to a finite collection of asymptotically finite-dimensional groups \cite{Osin2005}.  Theorem~\ref{thm:cubecomplexasdim} states that a finite-dimensional CAT(0) cube complex is asymptotically finite-dimensional, and implies that any cubulated group is asymptotically finite-dimensional.  More generally, Corollary~\ref{cor:asdimvertstab} gives conditions under which the hypothesis of properness of the action can be relaxed.  The following fundamental result was proven by Wright in~\cite{Wright2010}:

\begin{thm}\label{thm:cubecomplexasdim}
Let $\widetilde X$ be a CAT(0) cube complex.  Then $\asdim\widetilde X\leq\dimension\widetilde X$.
\end{thm}

Wright also observes that a finitely-generated group acting properly on a CAT(0) cube complex of dimension $D$ has asymptotic dimension at most $D$.  The main result of this section, Corollary~\ref{cor:asdimvertstab}, is a strengthening of Wright's result that also generalizes the theorem of Bell and Dranishnikov about graphs of asymptotically finite-dimensional groups.

Osin draws a striking contrast between relatively hyperbolic and weakly hyperbolic groups by giving examples of groups that are weakly hyperbolic relative to a finite collection of infinite cyclic subgroups but that contain free abelian groups of arbitrarily large rank and therefore have infinite asymptotic dimension.  Osin's groups are also quasi-arboreal relative to that collection of cyclic subgroups: the coset graph is bounded~\cite{Osin2005}.  On the other hand, these examples contain any recursively presentable group, and in particular have, for instance, subgroups with Property (T), and thus do not admit proper essential actions on CAT(0) cube complexes, by an application of a result in \cite{NibloRoller98}.



We now prove that the hypothesis of properness of the action of $G$ on $\widetilde X$ can be relaxed in Wright's result; one requires only uniform boundedness of the asymptotic dimension of 0-cube stabilizers.

\begin{cor}\label{cor:asdimvertstab}
Let $G$ be a finitely generated group acting on the locally finite CAT(0) cube complex $\widetilde X$, with $\dimension\widetilde X=D<\infty$.  Suppose there exists $n\in\naturals$ such that for each 0-cube $x$, the stabilizer $G_x$ satisfies $\asdim G_x\leq n$.  Then $\asdim G\leq n+D$.
\end{cor}

\begin{proof}
As usual, we will use the graph-metric on the 1-skeleton.  In particular, if $x$ is a 0-cube and $R\geq 0$, then $B_{\widetilde X}(x,R)$ denotes the smallest subcomplex of $\widetilde X$ containing all 0-cubes $y$ with $d_{\widetilde X}(x,y)\leq R$.

Let $x_o$ be a 0-cube of $\widetilde X$ and let $\psi:G\rightarrow\widetilde X^{(1)}$ be $\psi(g)=gx_o$.  This $\psi$ is a Lipschitz map with respect to the word metric on $G$ and the graph metric: the Lipschitz constant is $\max\{d_{\widetilde X}(x_o,sx_o):s\in\mathcal S\}$, where $\mathcal S$ is the finite generating set.  Let $R\geq 0$ and let $x=gx_o$ for some $g\in G$.  Then $B_{\widetilde X}(x,R)=\{gy_1,gy_2,\ldots,y_b\}$, where $y_1,\ldots,y_b$ are the finitely many 0-cubes at distance at most $R$ from $x_o$ ($b<\infty$ since $\widetilde X$ is locally finite).  The preimage of $B_{\widetilde X}(x,R)$ is therefore equal to $g\left(\cup_{i=1}^b\psi^{-1}(\{y_i\})\right)$.  Now, if $y_i\not\in Gx_o$, then $\psi^{-1}(\{y_i\})=\emptyset$.  Otherwise, $y_i=g_ix_o$ for some $g_i\in G$, whence $\psi^{-1}(y_i)=g_iG_{x_o}$.  Hence the preimage of $B_{\widetilde X}(x,R)$ is equal to $\cup_{i=1}^b(gg_i)G_{x_o}$, which is a finite union of cosets of $G$ for which $b$ depends on $R$ but not on $x$.  Therefore $\{\psi^{-1}(B(gx_o,R))\mid g\in G\}$ uniformly has asymptotic dimension bounded by $\asdim G_{x_o}$, which is at most $n$ by hypothesis.

Therefore, by the Hurewicz-type theorem~\cite{BellDranishnikov06}, $\asdim G\leq n+\asdim\psi(G)$.  By Theorem~\ref{thm:cubecomplexasdim}, $\asdim\widetilde X\leq D$.  On the other hand, $\asdim\psi(G)\leq\asdim\widetilde X$, and thus $\asdim G\leq n+D$, as required.

\end{proof}

\section{Hyperbolic cube complexes and complete bipartite subgraphs of $\Gamma$}\label{sec:planarbipartiteII}
The aim of this section is to characterize non-$\delta$-hyperbolic CAT(0) cube complexes in terms of the existence of certain complete bipartite subgraphs of their crossing- and contact graphs.  This leads, in a sense, to a combinatorial version of the ``flat plane theorem'' for cubulated groups.  Similar results are proved in~\cite{CDEHV}, from the point of view of median spaces.

Throughout this discussion, $\widetilde X$ is a CAT(0) cube complex with contact graph $\Gamma$ and crossing graph $\Delta$.

\subsection{Flat plane theorem}\label{subsec:flatplane}
\begin{defn}[Thin bicliques]
The graph $\Gamma$ has \emph{thin bicliques} if there exists $n\in\naturals$ such that any complete bipartite subgraph $K_{p,q}\subseteq\Gamma$ satisfies $p<n$ or $q<n$.
\end{defn}

The primary result is Theorem~\ref{thm:flatplaneonlyif}.  We use the following version of the axiom of choice.

\begin{lem}[K\"{o}nig's lemma]\label{lem:konig}
Let $\Lambda$ be a locally finite connected graph with infinitely many vertices and let $R$ be a subdivided ray.  Then for each vertex $v$, there is an embedding $R\hookrightarrow\Lambda$ containing $v$.
\end{lem}

\begin{thm}\label{thm:flatplaneonlyif}
Let $G$ be a group acting properly and cocompactly on the CAT(0) cube complex $\widetilde X$.
\begin{enumerate}
\item $G$ is word-hyperbolic if and only if the crossing graph $\Delta$ has thin bicliques.
\item If $G$ is not word-hyperbolic, then $\Delta$ contains the complete bipartite graph $K_{\infty,\infty}$.
\end{enumerate}
\end{thm}

We postpone the proof of Theorem~\ref{thm:flatplaneonlyif} until after that of Theorem~\ref{thm:thinbicliqueshyperbolic}, on which it depends, and also note that since $\Delta\subset\Gamma$, the complex $\widetilde X$ is hyperbolic if $\Gamma$ has thin bicliques.

\subsection{Hyperbolic CAT(0) cube complexes}
As usual, the graph $\widetilde X^{(1)}$, with metric $d_{\widetilde X}$, is \emph{$\delta$-hyperbolic} if for every geodesic triangle $\alpha_1\alpha_2\alpha_3\rightarrow\widetilde X^{(1)}$, each $\alpha_i$ lies in the $\delta$-neighborhood of the union of the other two segments.  The space $\widetilde X$ with the CAT(0) piecewise-Euclidean metric is $\delta'$-hyperbolic under the analogous condition on geodesic triangles.

The following lemma collects basic facts about hyperbolicity of cube complexes and the thin bicliques property of crossing graphs.

\begin{lem}\label{lem:dimensionhyper}
For a CAT(0) cube complex $\widetilde X$ with crossing graph $\Delta$, we have:
\begin{enumerate}
\item If $\widetilde X$ is finite-dimensional, then $\widetilde X$, with the usual CAT(0) metric, is hyperbolic if and only if $\widetilde X^{(1)}$ is a hyperbolic graph.
\item If $\widetilde X$ is infinite-dimensional, then it is not hyperbolic, and neither is $\widetilde X^{(1)}$.
\item If $\Delta$ has thin bicliques, then $\widetilde X$ is finite-dimensional.
\end{enumerate}
\end{lem}

\begin{proof}
\textbf{(1)} follows from the fact that a finite-dimensional CAT(0) cube complex is quasi-isometric to its 1-skeleton.\\
To prove \textbf{(2)}, note that for any $d\geq 0$, the existence of a $d$-cube guarantees the presence of a geodesic triangle, whose corners are 0-cubes, that is not $d$-thin.  Hence $\widetilde X$ is not $d$-thin for any $d$ if $\widetilde X$ contains arbitrarily large cubes.\\
If $\Delta$ has thin bicliques, then there is an upper bound on the cardinality of cliques in $\Delta$, since the existence of a complete subgraph on $2d$ vertices implies the existence of a complete $(d,d)$ bipartite subgraph.  The dimension of $\widetilde X$ is the maximal cardinality of cliques in $\Delta$, and \textbf{(3)} follows.
\end{proof}

When using disc diagrams, it is sometimes easier to think of a $\delta$-hyperbolic space as one whose isoperimetric inequality is linear than it is to verify the thin triangle condition.  Hence we shall sometimes rely on the following version of Gromov's characterization of hyperbolic metric spaces as those having linear isoperimetric inequality.  This result is stated in cubical terms as follows.

\begin{lem}[\cite{Gromov87}]\label{lem:isoperimetric}
Let $\widetilde X$ be a CAT(0) cube complex that is $\delta$-hyperbolic with respect to its CAT(0) metric.  There exists $\lambda\geq 0$ such that for each closed combinatorial path $\sigma\rightarrow\widetilde X$, there exists a disc diagram $D\rightarrow\widetilde X$ with $\partial_pD=\sigma$ such that the area of $D$ is at most $\lambda|\sigma|$.
\end{lem}

Actually, only the fact that the isoperimetric function of a hyperbolic metric space is subquadratic is invoked in our applications.

\subsection{Complete bipartite subgraphs of $\Delta$}\label{subsec:delta}
We first characterize hyperbolicity of the CAT(0) cube complex $\widetilde X$ in terms of complete bipartite subgraphs of the crossing graph $\Delta$.  Recall that the \emph{degree} of $\widetilde X$ is the supremum over all 0-cubes $x\in\widetilde X$ of the number of 1-cubes containing $x$.  The main result of this subsection is:

\begin{thm}\label{thm:thinbicliqueshyperbolic}
The finite-degree CAT(0) cube complex $\widetilde X$ is hyperbolic if and only if $\Delta$ has thin bicliques.
\end{thm}

Note that Theorem~\ref{thm:thinbicliqueshyperbolic} implies that $\widetilde X$ is hyperbolic when the contact-graph $\Gamma$ has thin bicliques.  Moreover, Lemma~\ref{lem:thinbicliquesimplieshyperbolic} does not require $\widetilde X$ to have finite degree, so any cube complex whose contact graph has thin bicliques is hyperbolic.  The proof of Theorem~\ref{thm:thinbicliqueshyperbolic} is assembled as follows from the lemmas below.

\begin{proof}[Proof of Theorem~\ref{thm:thinbicliqueshyperbolic}]
By Lemma~\ref{lem:thinbicliquesimplieshyperbolic}, $\widetilde X$ is hyperbolic when $\Delta$ has thin bicliques.  Conversely, if $\Delta$ does not have thin bicliques, then by Lemma~\ref{lem:bipartitesubgraphimpliesplanargrid}, $\widetilde X$ does not have a linear isoperimetric function and thus, by Lemma~\ref{lem:isoperimetric}, $\widetilde X$ is not $\delta$-hyperbolic for any $\delta$.
\end{proof}

In the presence of a proper, cocompact group action on $\widetilde X$, this yields the following:

\begin{proof}[Proof of Theorem~\ref{thm:flatplaneonlyif}]
$G$ is quasi-isometric to $\widetilde X$.  The first statement follows directly from Theorem~\ref{thm:thinbicliqueshyperbolic}, since $\widetilde X$ has finite degree.

If $G$ is not word-hyperbolic, then $\widetilde X$ is not word-hyperbolic, and thus $\Delta$ contains $K_{n,n}$ for all $n\geq 0$, by Theorem~\ref{thm:thinbicliqueshyperbolic}.  The proof of Lemma~\ref{lem:bipartitesubgraphimpliesplanargrid} shows that $\widetilde X$ therefore contains isometrically embedded copies of $[0,n]^2$ for any $n\geq 0$.  Note that for each $n\geq 0$, there are only finitely many $G$-orbits of such $n\times n$-planar grids in $\widetilde X$, because $\widetilde X$ is locally finite and $G$ acts cocompactly.  Indeed, each orbit of $n\times n$ planar grid is represented by one of the finitely many grids of the given dimensions that intersect a fixed compact fundamental domain for the $G$-action.

Let $\Lambda$ be the graph whose vertices correspond to $G$-orbits of isometrically embedded $n\times n$ planar grids in $\widetilde X$, for all $n\geq 0$.  The orbit represented by the $n\times n$ grid $D_n$ is adjacent in $\Lambda$ to the orbit represented by $D_{n+1}$ if and only if $D_n\subset gD_{n+1}$ for some $g\in G$, and this adjacency describes all edges in $\Lambda$.  Since $\widetilde X$ contains $D_n$ for all $n\geq 0$, the graph $\Lambda$ is infinite.  Furthermore, each vertex of $\Lambda$ is joined by a path to one of the finitely many vertices corresponding to $G$-orbits of $0\times0$ planar grids (i.e. 0-cubes).  Hence $\Lambda$ has finitely many components, at least one of which must be infinite.  Finally, $\Lambda$ is locally finite since there are finitely many orbits of planar grid of each size, and each $n\times n$ planar grid is adjacent in $\Lambda$ to vertices represented by $(n\pm1)\times(n\pm1)$ planar grids.  Hence $\Lambda$ contains an infinite ray, by K\"{o}nig's lemma, and thus there exists an infinite increasing union $D_0\subset D_1\subset\ldots$ of planar grids in $\widetilde X$.  Every hyperplane crossing $D_n$ crosses $D_m$ for $m>n$ and hence we have an increasing union
\[K_{0,0}\subset K_{1,1}\subset\ldots\subset K_{n,n}\subset K_{n+1,n+1}\subset\ldots\subset\Delta\]
of crossing graphs of the planar grids, whence $K_{\infty,\infty}\subseteq\Delta$.
\end{proof}

We now turn to the proof of Theorem~\ref{thm:thinbicliqueshyperbolic}.

\begin{defn}[Facing triple]\label{defn:communication}
The distinct hyperplanes $H_1,H_2,H_3$ form a \emph{facing triple} if any two lie in a single halfspace associated to the third.
\end{defn}

\begin{rem}[Planar grids from 4-cycles in $\Delta$]\label{rem:planargridfrom4cycle}
Let $H_0\bot V_0\bot H_1\bot V_1\bot H_0$ be an embedded 4-cycle in $\Delta$.  For $i\in\{0,1\}$, choose concatenable geodesic paths $P_i\rightarrow N(V_i),\,Q_i\rightarrow N(H_i)$ such that $A=P_0Q_0P_1Q_1$ is a closed path in $\widetilde X$.  Let $D\rightarrow\widetilde X$ be a disc diagram with boundary path $A$. Suppose that $A$ and $D$ are chosen so that $D$ has minimal complexity among all fixed carrier diagrams for the given 4-cycle in $\Delta$.  See Figure~\ref{fig:4cycle}.
\begin{figure}[h]
  \includegraphics[width=2.5in]{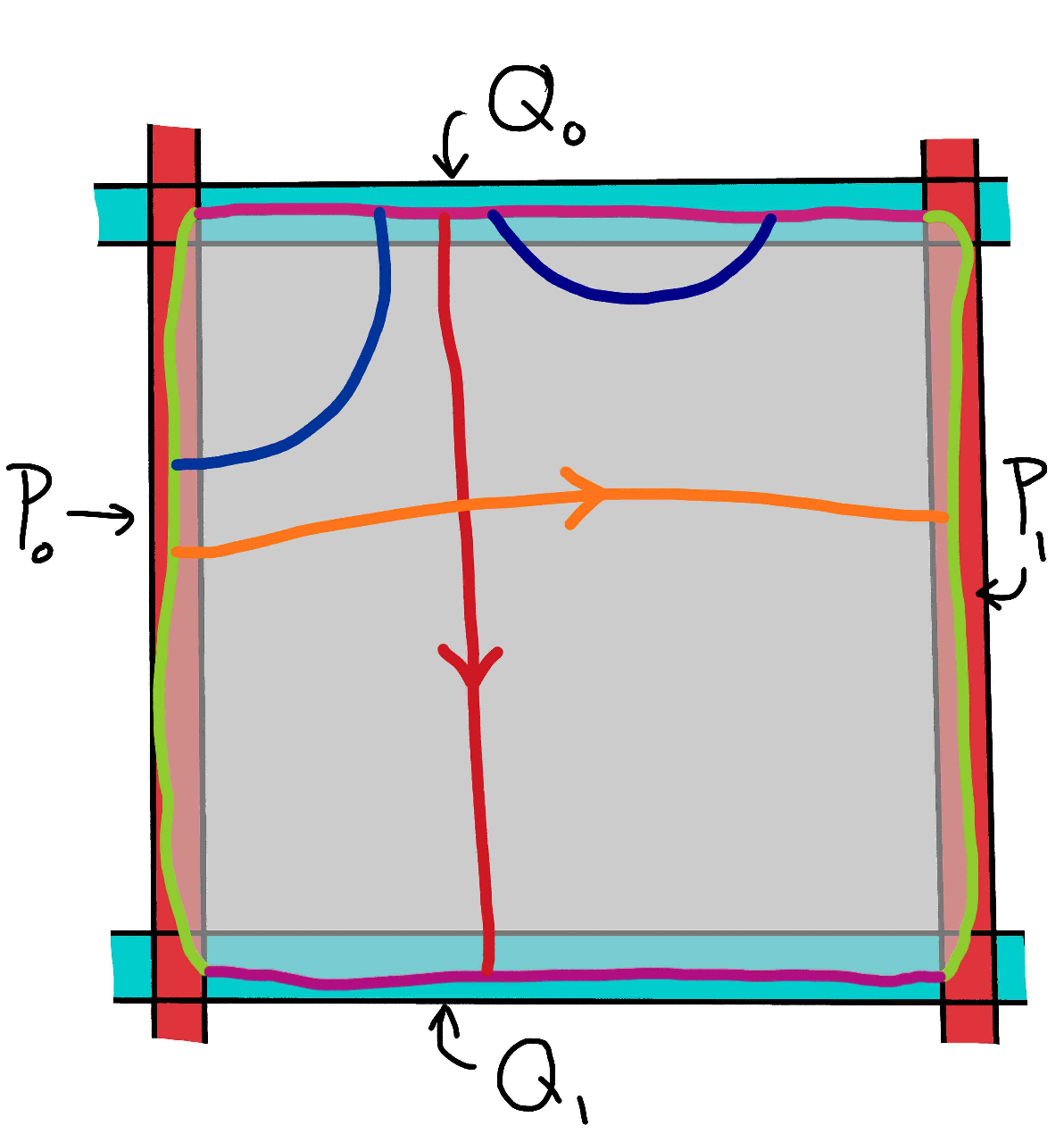}\\
  \caption{The diagram arising from a 4-cycle in $\Delta$.}\label{fig:4cycle}
\end{figure}

If $K$ is a dual curve in $D$, then $K$ travels from $P_0$ to $P_1$ or from $Q_0$ to $Q_1$.  Indeed, $K$ cannot travel from, say, $P_1$ to $P_1$, since $P_1$ is a geodesic.  By Lemma~\ref{lem:fixedcorners}, $K$ cannot travel from $P_i$ to $Q_j$, for then we could modify $A$, without affecting the 4-cycle in $\Delta$, to produce a lower complexity fixed-carrier diagram.  Similarly, no two dual curves emanating from $P_i$ (or $Q_i$) can cross.  Denote by $\mathbb H$ the set of dual curves traveling from $P_0$ to $P_1$ and by $\mathbb V$ the set of dual curves traveling from $Q_0$ to $Q_1$.  Each element of $\mathbb H$ crosses each element of $\mathbb V$, and there are no other intersections of dual curves in $D$.  Hence $D$ is a planar grid isomorphic to $P_0\times Q_0$.  In particular, $D$ is a CAT(0) cube complex whose set of hyperplanes is $\mathbb H\sqcup\mathbb V$.

Let $H,H'\in\mathbb H$ and $V\in\mathbb V$ be dual curves in $D$.  Since $H$ crosses $V$, the dual curves $H$ and $V$ map to distinct hyperplanes of $\widetilde X$, since hyperplanes in $\widetilde X$ do not self-cross.  Since $H$ and $H'$ are both dual to 1-cubes of $P_0$, and $P_0\rightarrow\widetilde X$ is a geodesic, the dual curves $H,H'$ must map to distinct hyperplanes.  Hence the map $D\rightarrow\widetilde X$ is injective on hyperplanes.  Since $D\rightarrow\widetilde X$ is a cubical map of CAT(0) cube complexes that is injective on hyperplanes, it is an isometric embedding.

Suppose $r\leq\min(|P_0|,|Q_0|)$.  Then $D$ contains an $r\times r$ planar grid $E$ with boundary path $P$.  Note that the map $D\rightarrow\widetilde X$ restricts to an isometric embedding $E\rightarrow\widetilde X$, and in particular $P$ embeds in $\widetilde X$.  Note that $|P|=4r$, while $|E|=r^2$.

If there exists a disc diagram $F\rightarrow\widetilde X$ with $\partial_pF=P$ and $\area(F)<r^2$, then we could excise the interior of $E$ from $D$ and attach $F$ along $P$ to obtain a lower-area diagram $D'$ with boundary path $A$, contradicting the fact that $D$ has minimal area among diagrams with boundary path $A$.  Hence every disc diagram bounded by $P$ has area at least $r^2$.
\end{rem}

\begin{lem}\label{lem:bipartitesubgraphimpliesplanargrid}
If $\widetilde X$ has finite degree and $\Delta$ does not have thin bicliques, then $\widetilde X$ is not $\delta$-hyperbolic for any $\delta<\infty$.
\end{lem}

\begin{proof}
If the degree $D$ of $\widetilde X$ is $0$, then $\widetilde X$ is a 0-cube.  If $D=1$, then $\widetilde X$ is a 1-cube.  If $D=2$, then either $\widetilde X$ is a single 2-cube or $\widetilde X$ is an interval.  In each of these cases, $\Delta$ has thin bicliques and $\widetilde X$ is hyperbolic.

If $D=3$, then by Remark~\ref{rem:planargridfrom4cycle}, $\Delta$ cannot contain an embedded 4-cycle and thus has thin bicliques.  On the other hand, either $\widetilde X$ is a single 3-cube, or $\widetilde X$ embeds in $T\times[-\frac{1}{2},\frac{1}{2}]$ for some tree $T$.  Hence $\widetilde X$ is hyperbolic.  Thus we assume that $D>3$.

For $2\ll R<\infty$, let $\mathcal H,\mathcal V$ be disjoint sets of hyperplanes, with $\min(|\mathcal H|,|\mathcal V|)\geq R$, such that $K(\mathcal V,\mathcal H)\subseteq\Delta$, i.e. for all $V\in\mathcal V,\,H\in\mathcal H$, we have $V\bot H$.  Let $V_0,V_1$ be distinct hyperplanes in $\mathcal V$ and let $H_0,H_1$ be distinct hyperplanes in $\mathcal H$.  Then $H_0\bot V_0\bot H_1\bot V_1\bot H_0$ is an embedded 4-cycle in $\Delta$.

Without loss of generality, $\mathcal V$ and $\mathcal H$ are inseparable.  Indeed, if $W$ is a hyperplane separating $H,H'\in\mathcal H$, then $W$ crosses each $V\in\mathcal V$, since $V\bot H$ and $V\bot H'$.  Hence we can include $W$ in $\mathcal H$ without affecting the fact that $\mathcal V$ and $\mathcal H$ generate a biclique in $\Delta$.

By Lemma~\ref{lem:productembed}, there exist isometrically embedded subcomplexes $A(\mathcal V)$ and $A(\mathcal H)$ such that the set of hyperplanes crossing $A(\mathcal V)$ is exactly $\mathcal V$ and the set of hyperplanes crossing $A(\mathcal H)$ is precisely $\mathcal H$.  By the same lemma, $\widetilde X$ contains an isometrically embedded subcomplex $A\cong A(\mathcal V)\times A(\mathcal H)$.

By Lemma~\ref{lem:interval}, for any $s\geq 0$, we can choose $R$ large enough that $A(\mathcal V)$ and $A(\mathcal H)$ respectively contain geodesic segments $P,Q$ of length at least $s$.  Hence $A\subset\widetilde X$ contains an $s\times s$ isometrically embedded planar grid $E$ with $|\partial_pE|=4s$.  By Remark~\ref{rem:planargridfrom4cycle}, and the fact that each of $P$ and $Q$ lies in a hyperplane-carrier, any disc diagram bounded by $\partial_pE$ has area at least $\area(E)=s^2$.  Thus $\widetilde X$ is not hyperbolic, by Lemma~\ref{lem:dimensionhyper} and Lemma~\ref{lem:isoperimetric}.
\end{proof}

\begin{rem}
The lemma holds for $\widetilde X$ of infinite maximal degree, under other interesting hypotheses.  For example, if $\widetilde X$ contains bicliques $K(\mathcal V,\mathcal H)$ such that $\mathcal V$ and $\mathcal H$ are free of facing triples, and can be chosen arbitrarily large, then $\widetilde X$ is not hyperbolic, by an argument very similar to that used to prove Lemma~\ref{lem:bipartitesubgraphimpliesplanargrid}.
\end{rem}

\begin{lem}\label{lem:productembed}
Let $\mathcal H$ be a finite, inseparable set of hyperplanes.  Then there exists an isometrically embedded compact subcomplex $A(\mathcal H)\subseteq\widetilde X$ such that the set of hyperplanes crossing $A(\mathcal H)$ is exactly $\mathcal H$.

Moreover, if $\mathcal V$ is another finite inseparable set, and for all $V\in\mathcal V,\,H\in\mathcal H$, the hyperplanes $H$ and $V$ cross, then there is an isometric embedding $A(\mathcal H)\times A(\mathcal V)\hookrightarrow\widetilde X$.
\end{lem}

\begin{proof}
Let $\mathcal H=\{H_1,\ldots,H_m\}$.  To produce $A(\mathcal H)$, we shall argue by induction on $m$.  In the base case, $m=1$, we choose $A(\mathcal H)$ to be any 1-cube dual to $H_1$.  Then $A(\mathcal H)$ is obviously compact and convex (and therefore isometrically embedded), and the unique hyperplane crossing $A(\mathcal H)$ is $H_1$.

Suppose that $m\geq 2$ and suppose that the above labeling of the elements of $\mathcal H$ has the property that $\{H_1,\ldots,H_{m-1}\}$ is inseparable.  Our induction hypothesis is that for any CAT(0) cube complex $\widetilde Y$ and any collection $\{W_1,\ldots,W_{m-1}\}$ of hyperplanes in $\widetilde Y$, there exists a compact convex subcomplex $A\subseteq\widetilde Y$ such that the set of hyperplanes crossing $A$ is precisely $\{H_1,\ldots,H_{m-1}\}$.  In particular, this is hypothesized for any convex subcomplex of $\widetilde X$, regarded as a CAT(0) cube complex in its own right.

Applying this hypothesis to $\widetilde X$ itself yields a compact, convex subcomplex $A_{m-1}\subset\widetilde X$ such that the set of hyperplanes that cross $A_{m-1}$ is exactly $\{H_1,\ldots,H_{m-1}\}$.\\

\textbf{The case $N(H_m)\cap A_{m-1}\neq\emptyset$:} First suppose that $N(H_m)\cap A_{m-1}$ contains a 0-cube $x$.  Since $x\in N(H_m)$, there exists a 1-cube $e$ dual to $H_m$ such that $x\in e$.  Let $A'=A_{m-1}\cup e$, and let $A(\mathcal H)$ be the cubical convex hull of $A'$.  More precisely, $A(\mathcal H)$ is the subcomplex whose 0-skeleton is the intersection of all halfspaces of $\widetilde X^{(0)}$ that contain $A'$.  By definition, each hyperplane crossing $A(\mathcal H)$ must cross $A'$, and therefore belongs to $\mathcal H$.  Conversely, each hyperplane of $\mathcal H$ crosses $A'$, and therefore $A(\mathcal H)$.  Indeed, the induction hypothesis ensures that $H_i$ crosses $A_{m-1}\subset A'$ for $1\leq i\leq m-1$, and $H_m$ crosses $e\subset A'$.  Being convex, $A(\mathcal H)$ is isometrically embedded, and, since finitely many hyperplanes cross $A(\mathcal H)$, it is compact.\\

\textbf{The general case:}  Consider the set $\mathcal A_{m-1}$ of all compact, convex subcomplexes of $\widetilde X$ that are crossed by exactly the set $\{H_1,\ldots,H_{m-1}\}$ of hyperplanes.  Among these, choose $A_{m-1}$ as close as possible to $N(H_m)$.  If $d_{\widetilde X}(N(H_m),A_{m-1})=0$, then the construction of $A(\mathcal H)$ is complete.  Hence suppose that some hyperplane $U$ separates $A_{m-1}$ from $H_m$.  Necessarily, $U\not\in\mathcal H$.

For all $j\in\{1,\ldots,H_{m-1}\}$, the hyperplane $U$ cannot separate $H_m$ from $H_j$, and $H_m$ and $A_{m-1}$ lie in distinct halfspaces associated to $U$, and $N(H_j)\cap A_{m-1}\neq\emptyset$.  It follows that $H_j\bot U$ for $1\leq j\leq m-1$.

Denote by $U_1$ the copy of $U$ bounding $N(U)\cong U\times[-1,1]$ on the side contained in the halfspace associated to $U$ that contains $H_m$.  Then $U_1$ is a CAT(0) cube complex whose hyperplanes have the form $V\cap U_1$, where $V$ is a hyperplane of $\widetilde X$ that crosses $U$.  Moreover, the map $V\mapsto V\cap U_1$ is a bijection from the set of hyperplanes of $\widetilde X$ that cross $U$ to the set of hyperplanes of $U_1$.  Since $H_j\bot U$ for each $j$, the cube complex $U_1$ contains a hyperplane $H_j\cap U_1$ for $1\leq j\leq m-1$.

Since $U_1$ is convex in $\widetilde X$, any two hyperplanes $H_i\cap U_1,H_j\cap U_1$ of $U_1$ are separated by a hyperplane $V\cap U_1$ of $U_1$ if and only if $H_i,H_j$ are separated in $\widetilde X$ by $V$.  Hence $\{H_1\cap U_1,\ldots,H_{m-1}\cap U_1\}$ is a set of $m-1$ inseparable hyperplanes in $U_1$.  Applying our induction hypothesis to $U_1$ shows that there exists a compact convex subcomplex $B_{m-1}\subset U_1$ such that the set of hyperplanes of $U_1$ that cross $B_{m-1}$ is precisely $\{H_1\cap U_1,\ldots,H_{m-1}\cap U_1\}$.

Now, the inclusion $U_1\hookrightarrow\widetilde X$ embeds $B_{m-1}$ in $\widetilde X$ as a compact subcomplex.  Moreover, since $B_{m-1}$ is convex in $U_1$, and $U_1$ is convex in $\widetilde X$, it follows that $B_{m-1}$ is convex in $\widetilde X$.

By construction, each $H_j$ crosses $B_{m-1}$.  Conversely, suppose that some hyperplane $V$ crosses $B_{m-1}$.  Then, regarding $B_{m-1}$ as a subcomplex of $U_1$, we see that $V\cap U_1=H_j\cap U_1$ for some $j$, and therefore that $V=H_j$.  We have verified that $B_{m-1}\in\mathcal A_{m-1}$.  Since $U_1$ contains $B_{m-1}$ and lies in the same halfspace associated to $U$ as does $H_m$, the subcomplexes $N(H_m)$ and $B_{m-1}$ of $\widetilde X$ are not separated by $U$.  Thus $B_{m-1}$ is an element of $\mathcal A_{m-1}$ that is strictly closer than $A_{m-1}$ to $N(H_m)$, contradicting our choice of $A_{m-1}$.  Hence we can always choose $A_{m-1}$ so that $A_{m-1}\cap N(H_m)\neq\emptyset$ and argue as above to produce $A(\mathcal H)$.\\

\textbf{Products:}  Let $\mathcal H,\mathcal V$ be finite, inseparable sets of hyperplanes such that each $H\in\mathcal H$ crosses each $V\in\mathcal V$.  Since $\mathcal V\cup\mathcal H$ is inseparable, the above argument yields a compact, convex subcomplex $A(\mathcal V\cup\mathcal H)$ such that the set of hyperplanes crossing $A(\mathcal V\cup\mathcal H)$ is precisely $V\cup\mathcal H$.  Similarly, we have convex, compact subcomplexes $A(\mathcal V),A(\mathcal H)$ such that the set of hyperplanes crossing $A(\mathcal V)$ (respectively, $A(\mathcal H)$) is exactly $\mathcal V$ (respectively, $\mathcal H$).

Let $b\in A(\mathcal V\cup\mathcal H)$ be a 0-cube.  (As usual, for a hyperplane $U$, we denote by $b(U)$ the associated halfspace containing $b$.)  Then $A(\mathcal V)$ contains a unique 0-cube $b_v$, called the \emph{projection} of $b$ to $A(\mathcal V)$, such that $b_v(V)=b(V)$ for all $V\in\mathcal V$, and $A(\mathcal H)$ contains a 0-cube $b_h$ such that $b_h(H)=b(H)$ for all $H\in\mathcal H$.  Indeed, to construct $b_v$, we choose a halfspace $b_v(U)$ for each hyperplane $U$ as follows.  First, $b_v(U)=b(U)$ when $U\in\mathcal V$.  Second, if $U\not\in\mathcal V\cup\mathcal H$ and $U$ does not separate $A(\mathcal V\cup\mathcal H)$ from $A(\mathcal V)$, then $b_v(U)=b(U)$.  Third, if $U\in\mathcal H$, or if $U$ separates $A(\mathcal V\cup\mathcal H)$ from $A(\mathcal V)$, then $U$ crosses each element of $\mathcal V$, and we can let $b_v(U)$ be the complement of $b(U)$ without affecting consistency.  Since $b_v$ differs from $b$ on finitely many hyperplanes, this orientation is canonical, and since the set of hyperplanes $U$ for which $b_v(U)\neq b(U)$ is precisely the set of hyperplanes separating $A(\mathcal V\cup\mathcal H)$ from $A(\mathcal V)$, we have $b_v\in A(\mathcal V)$.  The projection $b_h$ is constructed analogously.

Consider the map $A(\mathcal V\cup\mathcal H)^{(0)}\rightarrow A(\mathcal V)^{(0)}\times A(\mathcal H)^{(0)}$ defined by $b\mapsto (b_v,b_h)$.  For any $(b_v',b_h')\in A(\mathcal V)^{(0)}\times A(\mathcal H)^{(0)}$, let $b'$ be the 0-cube of $\widetilde X$ that orients hyperplanes not in $\mathcal V\cup\mathcal H$ toward a fixed 0-cube $b_o\in A(\mathcal V\cup\mathcal H)$, and orients $V\in\mathcal V$ toward $b_o$ if and only if $b'_v$ orients $V$ toward the projection of $b_o$, and orients $H\in\mathcal H$ toward $b_o$ if and only if $b'_h$ orients $H$ toward the projection of $b_o$.  Since elements of $\mathcal V$ can be oriented independently of elements of $\mathcal H$, this orientation is consistent, so $b'$ is a genuine 0-cube.  By construction, $b'$ maps to $(b'_v,b'_h)$, whence the given map is surjective.  Now, by construction, 0-cubes $b,b'\in A(\mathcal V\cup\mathcal H)$ differ exactly on those hyperplanes in $\mathcal V$ on which their projections to $A(\mathcal V)$ differ and on those hyperplanes in $\mathcal H$ on which their projections to $A(\mathcal H)$ differ.  Hence the given map is an isometric embedding, and it follows that $A(\mathcal V\cup\mathcal H)\cong A(\mathcal V)\times A(\mathcal H)$, as desired.
\end{proof}

\begin{lem}\label{lem:interval}
Let $\widetilde X$ be a CAT(0) cube complex of finite degree.  Then for all $s\geq 0$, there exists $R$ such that if $\mathcal H$ is a finite, inseparable set of hyperplanes with $|\mathcal H|\geq R$, then $A(\mathcal H)$ contains a geodesic segment of length at least $s$.
\end{lem}

\begin{proof}
Let $\mathcal H$ be a finite, inseparable set of hyperplanes, with $|\mathcal H|=R$.  Since $A(\mathcal H)$ is crossed by $R<\infty$ hyperplanes and is convex in $\widetilde X$, it is a compact CAT(0) cube complex.  Fix a base 0-cube $a_0\in A(\mathcal H)$ and let $\{a_1,\ldots,a_p\}$ be the finite set of 0-cubes $a_i$ such that no geodesic segment $\gamma$ joining $a_0$ to $a_i$ is properly contained in a geodesic segment $\gamma'$ emanating from $a_0$ (i.e. the $a_i$ are the 0-cubes ``as far as possible'' from $a_0$, for $i\geq 1$).  For $1\leq i\leq p$, let $\gamma_i$ be a geodesic segment in $A(\mathcal H)$ joining $a_0$ to $a_i$.

The set $\{\gamma_i\}_{i=1}^p$ can be chosen so that $\mathcal T=\cup_i\gamma_i$ is an embedded (though not necessarily isometrically embedded) tree in $A(\mathcal H)^{(1)}$.  Moreover, $\mathcal T$ contains at least one 1-cube dual to each hyperplane in $\mathcal H$.

We first show that $\mathcal T$ can be chosen to be a tree.  For any initial choice of $\{\gamma_i\}$, $\mathcal T$ is connected since each $\gamma_i$ contains $a_0$.  Now, if $\mathcal T$ contains a cycle, then for some $i\neq j$, there exist $\gamma_i'\subset\gamma_i,\gamma_j'\subset\gamma_j$ such that $\gamma_i'(\gamma_j')^{-1}$ is a closed path joining $a_0$ to a 0-cube $b\in\gamma_i\cap\gamma_j$.  By removing the interior of $\gamma_i'$ from $\gamma_i$, we can replace $\gamma_i$ with a geodesic $\gamma''_i$ joining $a_0$ to $a_i$ and consisting of $\gamma_j'$ followed by the subpath of $\gamma_i$ joining $b$ to $a_i$.  Indeed, either $\gamma''_j$ is a geodesic, or some hyperplane $H$ crosses $\gamma'_j$ and separates $b$ from $a_i$.  Thus $H$ separates $a_i$ from $b$ and $b$ from $a_0$.  Hence $a_0$ and $a_i$ lie in the same halfspace associated to $H$, i.e. $H$ does not separate $a_0$ from $a_i$.  Such an $H$ cannot cross a geodesic $\gamma_i$ joining $a_0$ to $a_i$, a contradiction.  Hence $\gamma''_j$ is a geodesic.

Let $\mathcal T'$ be the union of the paths $\{\gamma_k\}_{k\neq i}\cup\{\gamma''_i\}$.  $\mathcal T'$ is again a union of maximal geodesics in $A(\mathcal H)$, with one geodesic joining $a_0$ to each $a_i$.  But $\mathcal T'\subset\mathcal T$ was formed by breaking a cycle, and hence $\rank\left(\pi_1\mathcal T'\right)<\rank\left(\pi_1\mathcal T\right)$.  Since $\mathcal T$ is finite, repeating this procedure finitely many times yields the desired tree, which we henceforth denote by $\mathcal T$.

Since each $H\in\mathcal H$ separates at least two 0-cubes of $A(\mathcal H)$, it is obvious that $\mathcal T$ contains at least one 1-cube dual to each hyperplane, and hence a total of at least $R$ 1-cubes.

Next, let $D$ be the degree of $\widetilde X$.  We shall show that for some $i\leq p$, the path $\gamma_i$ has length at least $S=\log_{D-2}[(D-3)(R+1)+1]-1$.  Now, although $\mathcal T\hookrightarrow\widetilde X$ is not an isometric embedding, $\gamma_i$ is a geodesic of $\widetilde X$ since it is a geodesic of $A(\mathcal H)$ and the latter is isometrically embedded in $\widetilde X$.  Hence $A(\mathcal H)\subset\widetilde X$ contains a geodesic of length at least $S$.  Choosing $\mathcal H$ to have cardinality
\[R\geq\frac{\exp\left[(s+1)\log(D-2)\right]-1}{D-3}-1\]
ensures that $A(\mathcal H)$ contains a $\widetilde X$-geodesic of length at least $s$.

To conclude, consider $\mathcal T$ as a tree rooted at $a_0$.  Let $S=\max_i|\gamma_i|$, so that $S$ is the depth of $\mathcal T$.  Note that the depth of $\mathcal T$ is realized by a geodesic $\gamma_i$ of $\widetilde X$, and recall that $\mathcal T$ has at least $R$ edges (and $R+1$ vertices).  Let $d$ be the maximal degree of a vertex of $\mathcal T$.  In any case, $d\leq D$.  In Lemma~\ref{lem:bipartitesubgraphimpliesplanargrid}, $d\leq D-1$, since each vertex of $\mathcal T$ is contained in $A(\mathcal H)\times A(\mathcal V)$, and has at least one incident 1-cube dual to a hyperplane not in $\mathcal H$.  We also assume $d>2$, since otherwise $\mathcal T$ is a line segment with at least $R$ edges, and we have $S\geq R$.

View $\mathcal T$ as a subtree of a regular $d$-valent tree $\mathcal U$ rooted at $a_0$ having at least $R$ edges and the same depth as $\mathcal T$.  A computation shows that
\[R\leq\frac{(d-1)^{S+1}-1}{d-2}-1,\]
from which it follows that $S\geq\log_{d-1}[(R+1)(d-2)+1]-1$ as desired.
\end{proof}

\begin{rem}
While the planar grid arising in $\widetilde X$ from the complete bipartite graph $K\subset\Delta$ is isometrically embedded (in the combinatorial sense), it may not be convex.  Any distortion of $P$ in $\widetilde X$ reflects some failure of $K$ to be a full subgraph of $\Delta$, by Lemma~\ref{lem:convexsubwallspace}.
\end{rem}

\begin{lem}\label{lem:thinbicliquesimplieshyperbolic}
If $\Delta$ has thin bicliques, then $\widetilde X$ is hyperbolic.
\end{lem}

\begin{proof}
By Lemma~\ref{lem:dimensionhyper}, it suffices to show that $\widetilde X^{(1)}$ is $\delta$-hyperbolic for some $\delta$.

Suppose to the contrary that for any $n\in\naturals$, there exists a combinatorial geodesic triangle $\chi_n\eta_n\nu_n\rightarrow\widetilde X^{(1)}$ that is not $n$-thin.  Let $D_n\rightarrow\widetilde X$ be a disc diagram of minimal area with boundary path $\chi_n\eta_n\nu_n$.  By assumption, there exists a point $x\in\chi_n$ such that $d_{\widetilde X}(x,\eta_n\cup\nu_n)>n$.  Let $\mathcal V$ be the set of hyperplanes separating $x$ from $\nu_n$ and let $\mathcal H$ be the set of hyperplanes separating $x$ from $\eta_n$.  Let $\mathbb V$ be the set of dual curves in $D_n$ that separate $x$ from $\nu_n$ in $D_n$ and let $\mathbb H$ be the set of dual curves in $D_n$ separating $x$ from $\eta_n$.  The diagram $D_n$ is shown in Figure~\ref{fig:hyperbolicdisc}.
\begin{figure}[h]
  \includegraphics[width=1.8 in]{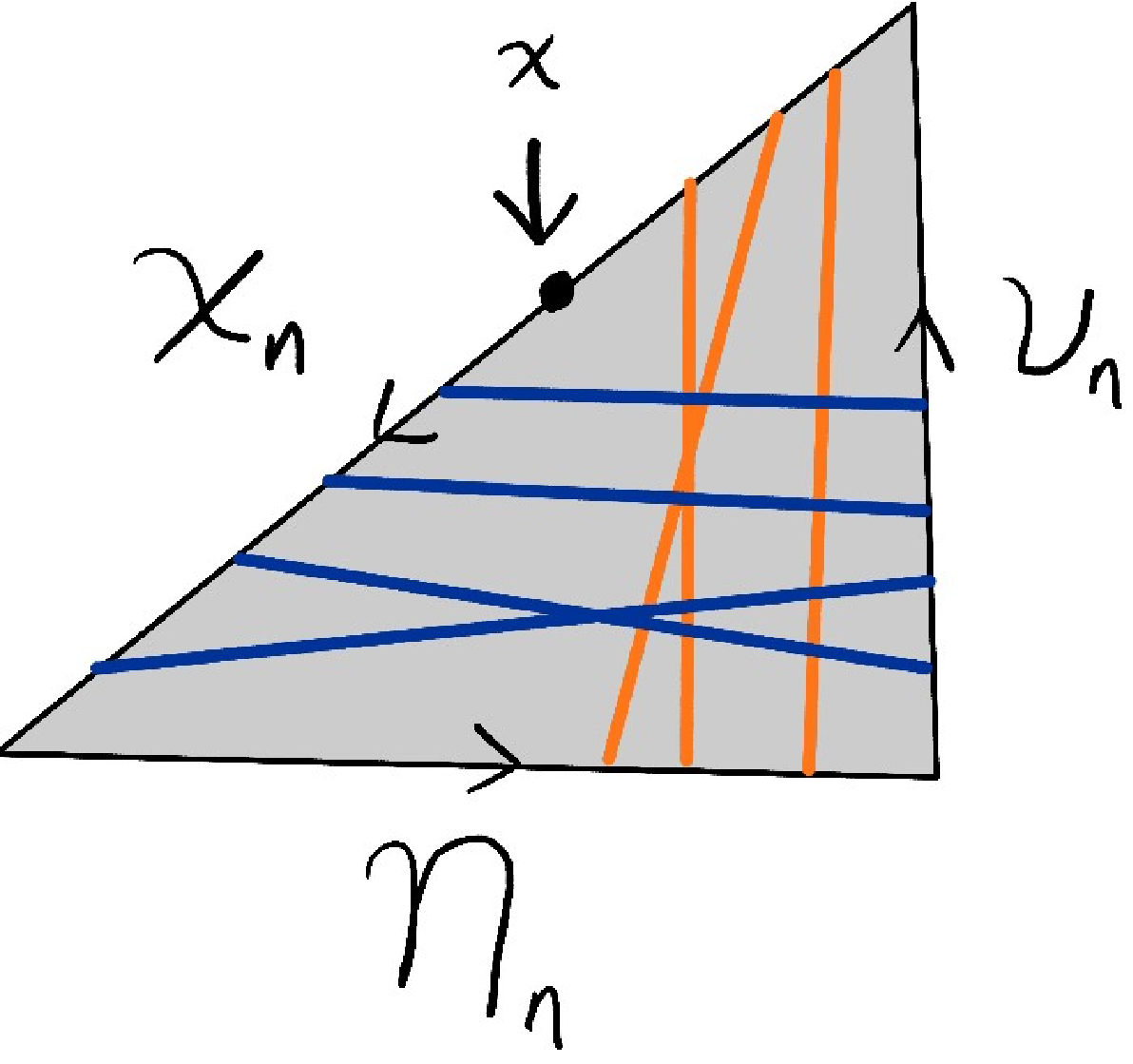}\\
  \caption{The diagram $D_n$ and a some vertical and horizontal separating dual curves.}\label{fig:hyperbolicdisc}
\end{figure}

We first show that the map $D_n\rightarrow\widetilde X$ induces bijections $\mathbb V\rightarrow\mathcal V$ and $\mathbb H\rightarrow\mathcal H$ and deduce that $|\mathbb V|,|\mathbb H|\geq n$.  A disc diagram argument then shows that each element of $\mathbb V$ crosses each element of $\mathbb H$ and thus that $K(\mathcal V,\mathcal H)$ is a complete bipartite subgraph of $\Delta$ with $|\mathcal V|,|\mathcal H|\geq n$.  Hence the failure of $\widetilde X$ to be hyperbolic implies that $\Delta$ does not have thin bicliques.\\
\textbf{The correspondences between $\mathbb V,\mathbb H$ and $\mathcal V,\mathcal H$:} Dual curves in $D_n$ map to distinct hyperplanes.  Indeed, since each side of the triangle $\partial_pD_n$ is a geodesic segment, no dual curve has both endpoints on the same side, because a geodesic contains at most a single 1-cube dual to each hyperplane.  Hence, if $C,C'$ are distinct dual curves in $D_n$, then one of the sides $\chi_n,\eta_n,\nu_n$ contains two of the four endpoints of $C\cup C'$.  Thus $C$ and $C'$ cannot map to the same hyperplane, for otherwise that side would cross a single hyperplane in two distinct 1-cubes, contradicting the fact that it is a geodesic.  Hence the maps $\mathbb V,\mathbb H\rightarrow\mathcal V,\mathcal H$ that associate dual curves in $D_n$ to hyperplanes according to the map $D_n\rightarrow\widetilde X$ are injective.

On the other hand, note that every element of $\mathbb V$ travels from $\chi_n$ to $\eta_n$.  Indeed, no dual curve in $D_n$ has both endpoints on the same side of the geodesic triangle.  Hence any $C\in\mathbb V$ travels from $\chi_n$ to $\eta_n$ since it cannot cross $\nu_n$ and similarly any $C\in\mathbb H$ travels from $\chi_n$ to $\nu_n$.  Any geodesic joining $x$ to some point of $\nu_n$ must cross each element of $\mathcal V$ exactly once, and thus each element of $\mathcal V$ occurs as a dual curve emanating from $\chi_n$ and terminating on $\eta_n$, i.e. as an element of $\mathbb V$.  The same argument holds for $\mathcal H$ and $\mathbb H$, and thus the desired correspondences between dual curves and hyperplanes are bijections.

Moreover, $|\mathcal V|,|\mathcal H|\geq n$, since the distance from $x$ to $\eta_n,\nu_n$ is precisely the number of hyperplanes separating $x$ from $\eta_n,\nu_n$.  Thus $|\mathbb V|,|\mathbb H|\geq n$.\\
\textbf{Crossing dual curves in $D_n$:}  Consider the decomposition $\chi_n=c_1c_2\ldots c_m$, where each $c_i$ is a 1-cube, with $c_1$ initial and $c_m$ terminal.  Suppose $x\in c_p$.  Then each element of $\mathcal V$ is dual to $c_i$ with $i\leq p$ and each element of $\mathcal H$ is dual to $c_i$ with $i\geq p$.  The dual curve emanating from $c_p$ belongs to $\mathbb V$, $\mathbb H$ or neither, according to the position of $x$ on $c_p$.  Since the elements of $\mathbb V$ end on $\eta_n$ and the elements of $\mathbb H$ end on $\nu_n$, each element of $\mathbb V$ crosses each element of $\mathbb H$ and hence $\mathcal V$ and $\mathcal H$ are the two classes of a complete bipartite subgraph of $\Delta$.
\end{proof}

\bibliographystyle{alpha}
\bibliography{MFHBIB}

%
%

\end{document}